\documentclass[11pt, reqno]{amsart}

\usepackage[x11names]{xcolor}
\usepackage{amsmath}
\usepackage{amssymb, amsthm}
\usepackage{hyperref}
\usepackage[shortlabels]{enumitem}
\usepackage[margin = 1in]{geometry}

\hypersetup{
    colorlinks = true,
    linktoc = all,
    urlcolor = Blue3!50!Blue4,
    linkcolor = Blue3!50!Blue4,
    citecolor = Blue3!50!Blue4
}

\newtheorem{theorem}{Theorem}[section]
\newtheorem{corollary}[theorem]{Corollary}
\newtheorem{lemma}[theorem]{Lemma}
\newtheorem{prop}[theorem]{Proposition}
\newtheorem{fact}[theorem]{Fact}
\newtheorem{claim}[theorem]{Claim}
\newtheorem{conjecture}[theorem]{Conjecture}

\newtheorem*{theorem*}{Theorem}
\newtheorem*{corollary*}{Corollary}
\newtheorem*{lemma*}{Lemma}
\newtheorem*{prop*}{Proposition}
\newtheorem*{fact*}{Fact}
\newtheorem*{claim*}{Claim}
\newtheorem*{conjecture*}{Conjecture}

\theoremstyle{definition}

\newtheorem{remark}[theorem]{Remark}

\newtheorem*{example*}{Example}
\newtheorem*{defn*}{Definition}
\newtheorem*{remark*}{Remark}
\newtheorem*{algo*}{Algorithm}

\numberwithin{equation}{section}
\newcommand{\CC}{\mathbb C}

\newcommand{\RR}{\mathbb R}

\newcommand{\ZZ}{\mathbb Z}

\newcommand{\cC}{\mathcal C}

\newcommand{\cP}{\mathcal P}
\newcommand{\cQ}{\mathcal Q}
\newcommand{\cR}{\mathcal R}
\newcommand{\cS}{\mathcal S}
\newcommand{\cT}{\mathcal T}

\newcommand{\cZ}{\mathcal Z}

\newcommand{\eps}{\varepsilon}

\newcommand{\sfloor}[1]{\lfloor #1 \rfloor}

\newcommand{\abs}[1]{\left\lvert #1 \right\rvert}
\newcommand{\sabs}[1]{\lvert #1 \rvert}

\newcommand{\term}[1]{(\mathrm{#1})}

\newcommand{\blank}{\,\cdot\,}

\title{Expanding polynomials for sets with additive structure}
\author{Sanjana Das}
\address{Department of Mathematics, Massachusetts Institute of Technology, MA, USA}
\email{\href{mailto:sanjanad@mit.edu}{sanjanad@mit.edu}}
\author{Cosmin Pohoata}
\address{Department of Mathematics, Emory University, GA, USA}
\email{\href{mailto:cosmin.pohoata@emory.edu}{cosmin.pohoata@emory.edu}}
\author{Adam Sheffer}
\address{Department of Mathematics, Baruch College, City University of New York, NY, USA}
\email{\href{mailto:adamsh@gmail.com}{adamsh@gmail.com}}
\date{October 27, 2024}
\thanks{This project was conducted as part of the 2024 NYC Discrete Math REU, funded by NSF awards DMS-2051026 and DMS-2349366 and by Jane Street. CP's work is supported by NSF award DMS-2246659.}

\begin{document}

\begin{abstract}
    The expansion of bivariate polynomials is well-understood for sets with a linear-sized product set. In contrast, not much is known for sets with small sumset. In this work, we provide expansion bounds for polynomials of the form $f(x, y) = g(x + p(y)) + h(y)$ for sets with small sumset. In particular, we prove that when $\abs{A}$, $\abs{B}$, $\abs{A + A}$, and $\abs{B + B}$ are not too far apart, for every $\eps > 0$ we have \[\abs{f(A, B)} = \Omega\left(\frac{\abs{A}^{256/121 - \eps}\abs{B}^{74/121 - \eps}}{\abs{A + A}^{108/121}\abs{B + B}^{24/121}}\right).\] We show that the above bound and its variants have a variety of applications in additive combinatorics and distinct distances problems. 

    Our proof technique relies on the recent proximity approach of Solymosi and Zahl. In particular, we show how to incorporate the size of a sumset into this approach. 
\end{abstract}

\maketitle

\section{Introduction}\label{sec:intro}

\subsection{Background}

Most multivariate polynomials $f$ have the following curious property: Given any set $A \subseteq \RR$, if we evaluate $f$ on all inputs drawn from $A$, the resulting set of outputs is much larger than $A$. For example, for the polynomial $f(x, y, z) = x + yz$, any finite set $A \subseteq \RR$ satisfies 
\begin{equation}
    \abs{\{f(a_1, a_2, a_3) \mid a_1, a_2, a_3 \in A\}} = \Omega(\abs{A}^{3/2}),\label{eqn:aa-a}
\end{equation} 
and it is conjectured that the exponent $\frac{3}{2}$ can be replaced with $2 - \eps$ for any $\eps > 0$ (see \cite{RNRSS19} and \cite{SW22} for slightly better bounds and further discussion). Many problems can be reduced to studying such expansion properties of polynomials, and the study of these properties is currently an active and exciting research front; for example, see \cite{RSdZ15, Tao15, RSS16, BB21}.

More formally, for a polynomial $f \in \RR[x_1, \ldots, x_k]$ and sets $A_1, \ldots, A_k \subseteq \RR$, we write \[f(A_1, \ldots, A_k) = \{f(a_1, \ldots, a_k) \mid a_1 \in A_1, \ldots, a_k \in A_k\}.\] We say $f$ is \emph{expanding} if for any sets $A_1, \ldots, A_k \subseteq \RR$ of size $n$, the set $f(A_1, \ldots, A_k)$ has size asymptotically larger than $n$ (which we write as $\abs{f(A_1, \ldots, A_k)} = \omega(n)$). 

In this paper, we focus on the case of bivariate polynomials $f \in \RR[x, y]$. In \cite{ER00}, Elekes and R\'onyai fully characterized \emph{which} bivariate polynomials are expanding. 

\begin{theorem}[Elekes--R\'onyai]\label{thm:er00}
    Let $f \in \RR[x, y]$ be a polynomial, and suppose that there do not exist polynomials $p, q \in \RR[x]$ for which
    \begin{equation}
        f(x, y) = g(p(x) + q(y)) \quad \text{or} \quad f(x, y) = g(p(x) \cdot q(y)).\label{eqn:el-ron-special}
    \end{equation}
    Then for all sets $A, B \subseteq \RR$ with $\abs{A} = \abs{B} = n$, we have $\abs{f(A, B)} = \omega(n)$. 
\end{theorem}

The constant in the asymptotic notation of Theorem \ref{thm:er00}, as well as in similar bounds throughout the paper, depends on $\deg f$.

The condition in Theorem \ref{thm:er00} that $f$ is not of the special form $f(x, y) = g(p(x) + q(y))$ or $f(x, y) = g(p(x) \cdot q(y))$ is necessary --- if $f$ were of the first special form, then we could construct sets $A$ and $B$ of size $n$ for which $\abs{f(A, B)} = \Theta(n)$ by taking $p(A)$ and $q(B)$ to be arithmetic progressions with the same difference. Similarly, if $f$ were of the second special form, then we could take $p(A)$ and $q(B)$ to be \emph{geometric} progressions. So polynomials of these special forms are not expanding, which means Theorem \ref{thm:er00} fully answers the qualitative question of which polynomials are expanding. 

It is then natural to ask for \emph{quantitative} lower bounds on $\abs{f(A, B)}$ when we know $f$ is expanding. The first such bound was proven by Raz, Sharir, and Solymosi \cite{RSS16}.

\begin{theorem}[Raz--Sharir--Solymosi]\label{thm:rss16}
    Let $f \in \RR[x, y]$ be a polynomial which is not of either special form in \eqref{eqn:el-ron-special}. Then for all sets $A, B \subseteq \RR$, we have \[\abs{f(A, B)} = \Omega(\min\{\abs{A}^{2/3}\abs{B}^{2/3}, \abs{A}^2, \abs{B}^2\}).\] 
\end{theorem}

This was then improved by Solymosi and Zahl \cite{SZ24} to the following bound. 
\begin{theorem}[Solymosi--Zahl]\label{thm:sz24}
    Let $f \in \RR[x, y]$ be a polynomial which is not of either special form in \eqref{eqn:el-ron-special}. Then for all sets $A, B \subseteq \RR$, we have \[\abs{f(A, B)} = \Omega(\min\{\abs{A}^{3/4}\abs{B}^{3/4}, \abs{A}^2, \abs{B}^2\}).\] 
\end{theorem}

In particular, when $\abs{A} = \abs{B} = n$, Theorem \ref{thm:sz24} gives a bound of $\abs{f(A, B)} = \Omega(n^{3/2})$. For more information on expanding bivariate polynomials and their applications, see \cite{dZ18}.

\subsection{Our expansion results}

In this paper, we consider the following question: Can we prove stronger expansion properties if we know that $A$ and $B$ are `additively structured'? To make this precise, we define the \emph{sumset} of a set $A \subseteq \RR$ as \[A + A = \{a + a' \mid a, a' \in A\}.\] Intuitively, a set with a small sumset is similar to a generalized arithmetic progression (for more details, see \cite[Chapter 7]{Zha23}). For this reason, we say that such sets are additively structured. So we want to prove stronger lower bounds on $\abs{f(A, B)}$ when $A + A$ and $B + B$ are small. 

Throughout this paper, we work over $\RR$. We work only with polynomials of the form \[f(x, y) = g(x + p(y)) + h(y)\] for single-variable polynomials $g$, $h$, and $p$. Although this is a rather restrictive form, it encompasses the polynomials that arise from at least two natural problems --- the problem of showing that $A$ and $g(A)$ cannot both have small sumsets (where we consider $f(x, y) = g(x) + g(y)$), and the problem of bounding the number of distinct distances between two lines (where if our lines are the $x$-axis and $x = sy$, we consider $f(x, y) = (x - sy)^2 + y^2$). We derive new results for both problems, which we discuss in Subsections \ref{subsec:convex} and \ref{subsec:distances}. 

\begin{remark}
    Many of the applications that we will discuss involve the case where $p = 0$, meaning that $f$ is of the form \[f(x, y) = g(x) + h(y)\] for some $g, h \in \RR[x]$. These polynomials are \emph{not} expanding in general, as they are of the first special form in \eqref{eqn:el-ron-special}. But the sets $A$ and $B$ in the construction we used to show they are not expanding do not have additive structure when $\deg g, \deg h \geq 2$. So unlike in results such as Theorems \ref{thm:er00}, \ref{thm:rss16}, and \ref{thm:sz24}, we can still hope to get lower bounds on $\abs{f(A, B)}$ for such polynomials. 
\end{remark}

Our first result considers the setting where only the set $A$ has additive structure. 

\begin{theorem}\label{thm:expand-one-set}
    Let $f \in \RR[x, y]$ be of the form $f(x, y) = g(x + p(y)) + h(y)$ where $g$, $p$, and $h$ are single-variable polynomials, $\deg g \geq 2$, and $\deg h \geq 1$. Then for every $\eps > 0$, there is some $c > 0$ such that for any $A, B \subseteq \RR$ with $\abs{A + A} \leq c\abs{A}^{3/2}$, we have \[\abs{f(A, B)} = \Omega\left(\min\left\{\frac{\abs{A}^{28/13 - \eps}\abs{B}^{4/13 - \eps}}{\abs{A + A}^{12/13}}, \frac{\abs{A}^{8/3}\abs{B}^{1/3 - \eps}}{\abs{A + A}^{4/3}}, \abs{A}^2, \abs{A}\abs{B}\right\}\right)\] (where both $c$ and the implicit constant depend on $\deg f$ and $\eps$). 
\end{theorem}

In the regime where $\abs{A} = \abs{B} = n$ and $\abs{A + A} = \Theta(n)$ (intuitively, this means that $A$ has the maximum possible additive structure), Theorem \ref{thm:expand-one-set} gives a bound of $\abs{f(A, B)} = \Omega(n^{20/13 - 2\eps})$. Meanwhile, when $A + A$ is reasonably small and $B$ is much smaller than $A$ (e.g., $\abs{A + A} \leq \abs{A}^{9/8}$ and $\abs{B} \leq \abs{A}^{1/8}$), Theorem \ref{thm:expand-one-set} implies that $\abs{f(A, B)} = \Omega(\abs{A}\abs{B})$. This is tight, as we always have $\abs{f(A, B)} \leq \abs{A}\abs{B}$.  

We prove a stronger bound when \emph{both} $A$ and $B$ have additive structure. 

\begin{theorem}\label{thm:expand-two-sets}
    Let $f \in \RR[x, y]$ be of the form $f(x, y) = g(x + p(y)) + h(y)$ where $g$, $p$, and $h$ are single-variable polynomials and $\deg g, \deg h \geq 2$. Then for every $\eps > 0$, there is some $c > 0$ such that for any $A, B \subseteq \RR$ with $\abs{A + A} \leq c\abs{A}^{3/2}$ and $\abs{B + B} \leq \abs{B}^{4/3}$, we have \[\abs{f(A, B)} = \Omega\left(\min\left\{\frac{\abs{A}^{256/121 - \eps}\abs{B}^{74/121 - \eps}}{\abs{A + A}^{108/121}\abs{B + B}^{24/121}}, \frac{\abs{A}^{74/29}\abs{B}^{28/29 - \eps}}{\abs{A + A}^{36/29}\abs{B + B}^{12/29}}, \abs{A}^2, \abs{A}\abs{B}\right\}\right)\] (where both $c$ and the implicit constant depend on $\deg f$ and $\eps$).  
\end{theorem}

In the regime where $\abs{A} = \abs{B} = n$ and $\abs{A + A}$ and $\abs{B + B}$ are both $\Theta(n)$, Theorem \ref{thm:expand-two-sets} gives a bound of $\abs{f(A, B)} = \Omega(n^{18/11 - 2\eps})$. 

\begin{remark}\label{rmk:sumset-diffset}
    Similarly to the sumset of $A$, the \emph{difference set} of $A$ is defined as \[A - A = \{a - a' \mid a, a' \in A\}.\] Theorems \ref{thm:expand-one-set} and \ref{thm:expand-two-sets} also hold if we replace all the sumsets with difference sets; the same proofs work, with straightforward modifications.
\end{remark} 

\subsection{Applications to sum-product type bounds}\label{subsec:convex}

We now discuss some applications of Theorems \ref{thm:expand-one-set} and \ref{thm:expand-two-sets}. 

The sum-product problem concerns the statement that a set $A \subseteq \RR$ cannot be both additively and multiplicatively structured. To make this precise, we define the \emph{product set} of $A$ as \[AA = \{aa' \mid a, a' \in A\}.\] Similarly to the additive case, we think of $A$ as multiplicatively structured if its product set is small. Erd\H{o}s and Szemer\'edi \cite{ES83} conjectured that for any set $A$, either its sumset or product set must be of near-maximal size. 

\begin{conjecture}[Erd\H{o}s--Szemer\'edi]\label{conj:sum-prod}
    Let $\eps > 0$. Then for any set $A \subseteq \ZZ$, we have \[\max\{\abs{A + A}, \abs{A A}\} = \Omega(\abs{A}^{2 - \eps}).\] 
\end{conjecture}

Erd\H{o}s and Szemer\'edi \cite{ES83} proved that $\max\{\abs{A + A}, \abs{AA}\} = \Omega(\abs{A}^{1 + \delta})$ for some $\delta > 0$. Several authors have improved the value of $\delta$ and generalized the result from $\ZZ$ to $\RR$; the current best bound is due to Rudnev and Stevens \cite{RS22}, who showed that we can take $\delta = \frac{1}{3} + \frac{2}{1167} - \eps$ for any $\eps > 0$. 

One commonly studied variant of this problem concerns the statement that for a strictly convex function $g$, the sets $A$ and $g(A)$ cannot both be additively structured --- more precisely, $A + A$ and $g(A) + g(A)$ cannot both be small. (The case $g(x) = -\log x$ corresponds to the sum-product problem.) The current best bound for this problem is due to Stevens and Warren \cite{SW22}. (For definitions of the asymptotic notation $\Omega^*(\blank)$ and $O^*(\blank)$, see Section \ref{sec:prelims}.)

\begin{theorem}[Stevens--Warren]\label{thm:sw22}
    Let $g \colon \RR \to \RR$ be convex. Then for all sets $A \subseteq \RR$, we have \[\abs{g(A) + g(A)} = \Omega^*\left(\frac{\abs{A}^{49/19}}{\abs{A + A}}\right).\] In particular, for all $A \subseteq \RR$, we have $\max\{\abs{A + A}, \abs{g(A) + g(A)}\} = \Omega^*(\abs{A}^{49/38})$. 
\end{theorem}

Theorem \ref{thm:expand-two-sets} leads to the following improved bound when $g$ is a polynomial of degree at least $2$. 

\begin{corollary}\label{cor:sumset-poly}
    Let $g \in \RR[x]$ be a polynomial with $\deg g \geq 2$, and let $\eps > 0$. Then for all sets $A \subseteq \RR$ such that $\abs{A + A} \leq \abs{A}^{4/3}$, we have \[\abs{g(A) + g(A)} = \Omega\left(\frac{\abs{A}^{30/11 - \eps}}{\abs{A + A}^{12/11}}\right).\] In particular, for all $A \subseteq \RR$, we have $\max\{\abs{A + A}, \abs{g(A) + g(A)}\} = \Omega(\abs{A}^{30/23 - \eps})$. 
\end{corollary}

Lower bounds on $\max\{\abs{A + A}, \abs{g(A) + g(A)}\}$ are one way of quantifying the notion that $g$ `destroys additive structure.' Another way of quantifying this is by considering the case where $A + A$ is minimal and analyzing how big $g(A) + g(A)$ must be. This has been studied from the perspective of convex sets. A set $C = \{c_1, \ldots, c_n\} \subseteq \RR$ is \emph{convex} if \[c_2 - c_1 < c_3 - c_2 < \cdots < c_n - c_{n - 1}.\] Equivalently, $C$ is a convex set if and only if $C = g([n])$ for a convex function $g$. The problem of proving that convex sets must have large sumsets have been studied by many authors; Schoen and Shkredov \cite{SS11} proved that $\abs{C + C} = \Omega^*(\abs{C}^{14/9})$. 

In the extreme case where $\abs{A + A} = \Theta(\abs{A})$, Theorem \ref{thm:sw22} of \cite{SW22} gives \[\abs{g(A) + g(A)} = \Omega^*(\abs{A}^{30/19}),\] which in particular means $\abs{C + C} = \Omega^*(\abs{C}^{30/19})$ for any convex set $C$ (this slightly improves the bound of \cite{SS11}). Meanwhile, when $g$ is a polynomial, Corollary \ref{cor:sumset-poly} gives \[\abs{g(A) + g(A)} = \Omega(\abs{A}^{18/11 - \eps}).\] This is a stronger bound, but for a less general class of functions $g$. 

Finally, several authors have proved more general lower bounds on $\max\{\abs{A + A}, \abs{g(A) + B}\}$ for arbitrary sets $B$. In particular, Shkredov \cite{Shk15} proved the following bound. (Before \cite{SW22}, this gave the best lower bound on $\max\{\abs{A + A}, \abs{g(A) + g(A)}\}$.) 

\begin{theorem}[Shkredov]
    Let $g \colon \RR \to \RR$ be continuous and convex. Then for all sets $A, B \subseteq \RR$ with $\abs{A} = \abs{B}$, we have \[\abs{g(A) + C} = \Omega^*\left(\frac{\abs{A}^{50/21}}{\abs{A + A}^{37/42}}\right).\] 
\end{theorem}

When $g$ is a polynomial, if we apply Theorem \ref{thm:expand-one-set} with $f(x, y) = g(x) + y$, then we get the following result, which is an improvement when $\abs{A + A} \leq \abs{A}^{26/19 - \eps}$ for any fixed $\eps > 0$. 

\begin{corollary}\label{cor:asym-one-set}
    Let $g \in \RR[x]$ be a polynomial with $\deg g \geq 2$, and let $\eps > 0$. Then for all sets $A, B \subseteq \RR$ with $\abs{A} = \abs{B}$, we have \[\abs{g(A) + B} = \Omega\left(\min\left\{\frac{\abs{A}^{32/13 - \eps}}{\abs{A + A}^{12/13}}, \frac{\abs{A}^{3 - \eps}}{\abs{A + A}^{4/3}}\right\}\right).\] 
\end{corollary}

As opposed to Theorem \ref{thm:expand-one-set}, Corollary \ref{cor:asym-one-set} does not include the condition $\abs{A + A} \leq c\abs{A}^{3/2}$. This is because when $\abs{A + A} = \Omega(\abs{A}^{3/2})$, the second term in Corollary \ref{cor:asym-one-set} is $O(\abs{A})$, which means the bound of Corollary \ref{cor:asym-one-set} is still true in this case. 

\subsection{Applications to distinct distances problems}\label{subsec:distances}

Next, we discuss some applications of Theorems \ref{thm:expand-one-set} and \ref{thm:expand-two-sets} to distinct distances problems. For two points $p, q \in \RR^2$, we write $d(p, q)$ to denote the distance between them. For a set of points $\cP \subseteq \RR^2$, we write \[\Delta(\cP) = \{d(p, q) \mid p, q \in \cP\}.\] Similarly, for two sets of points $\cP_1, \cP_2 \subseteq \RR$, we write \[\Delta(\cP_1, \cP_2) = \{d(p, q) \mid p \in \cP_1, q \in \cP_2\}.\]

The study of distinct distances started with the following question by Erd\H{o}s \cite{Erd46}: What is the minimum possible number of distinct distances that $n$ points in the plane can form? In other words, what is the minimum value of $\abs{\Delta(\cP)}$ over all sets $\cP$ of $n$ points? Erd\H{o}s showed that when $\cP$ is a $\sqrt{n} \times \sqrt{n}$ lattice, we have $\abs{\Delta(\cP)} = O(\frac{n}{\sqrt{\log n}})$. The problem of proving a comparable \emph{lower} bound on $\abs{\Delta(\cP)}$ was studied by many authors; finally, Guth and Katz \cite{GK15} proved that any set $\cP$ of size $n$ satisfies $\abs{\Delta(\cP)} = \Omega(\frac{n}{\log n})$, which almost matches Erd\H{o}s's upper bound. 

While the above question is almost fully resolved, most distinct distances problems are still wide open. We now discuss some applications of our results to two of these problems. 

\subsubsection{Distinct distances between two lines}

The first problem we consider is about the minimum number of distinct distances between points on two lines: Let $\ell_1$ and $\ell_2$ be two lines, and let $\cP_1$ and $\cP_2$ be finite sets of points on $\ell_1$ and $\ell_2$, respectively. How small can $\Delta(\cP_1, \cP_2)$ be? This question is one of the most elementary distinct distances problems, in the sense that each point is defined by a single coordinate. However, it is still far from solved; this makes it a central distinct distances problem. 

We have $\sabs{\Delta(\cP_1, \cP_2)} = \Omega(\abs{\cP_1} + \abs{\cP_2})$. When $\ell_1$ and $\ell_2$ are parallel or orthogonal, this is tight. Indeed, when $\ell_1 \parallel \ell_2$, we can take equally spaced points on the two lines; when $\ell_1 \perp \ell_2$, we can take $\cP_1 = \{(\sqrt{a}, 0) \mid a \in [m]\}$ and $\cP_2 = \{(0, \sqrt{b}) \mid b \in [n]\}$. However, when $\ell_1$ and $\ell_2$ are neither parallel nor orthogonal, the number of distances is significantly larger; the current best lower bound is due to Solymosi and Zahl \cite{SZ24}. 

\begin{theorem}[Solymosi--Zahl]\label{thm:sz-distances-lines}
    Suppose that $\ell_1$ and $\ell_2$ are neither parallel nor orthogonal, and let $\cP_1 \subseteq \ell_1$ and $\cP_2 \subseteq \ell_2$ be finite sets of points. Then \[\abs{\Delta(\cP_1, \cP_2)} = \Omega(\min\{\abs{\cP_1}^{3/4}\abs{\cP_2}^{3/4}, \abs{\cP_1}, \abs{\cP_2}\}).\] 
\end{theorem}

The problem of lower-bounding $\Delta(\cP_1, \cP_2)$ is an expanding polynomials problem. Indeed, let $\ell_1$ be the line $y = 0$ and $\ell_2$ the line $x = sy$ (this setup allows the lines to be orthogonal, but not parallel). Let $A$ be the set of $x$-coordinates of $\cP_1$ and $B$ the set of $y$-coordinates of $\cP_2$. Then the polynomial \[f(x, y) = (x - sy)^2 + y^2\] is the square of the distance between the points $(x, 0)$ and $(sy, y)$, which means $\sabs{\Delta(\cP_1, \cP_2)} = \abs{f(A, B)}$. Theorem \ref{thm:sz-distances-lines} immediately follows from applying Theorem \ref{thm:sz24} to this polynomial $f$ (if $s \neq 0$, then $f$ does not have either special form in \eqref{eqn:el-ron-special}). 

Using Theorems \ref{thm:expand-one-set} and \ref{thm:expand-two-sets}, we can get improved bounds in the case where $A$ and $B$ have additive structure, which also hold when $\ell_1$ and $\ell_2$ are orthogonal. For these bounds, we use the variants of Theorems \ref{thm:expand-one-set} and \ref{thm:expand-two-sets} with difference sets rather than sumsets (see Remark \ref{rmk:sumset-diffset}) because difference sets have a simpler geometric interpretation here --- we have $\abs{A - A} = 2\abs{\Delta(\cP_1)} - 1$ and $\abs{B - B} = 2\abs{\Delta(\cP_2)} - 1$. So we can view these bounds as stating that if there are few distinct distances \emph{within} $\cP_1$ and $\cP_2$ (relative to their sizes), then there are many distinct distances \emph{between} them. 

To simplify the bounds, we restrict our attention to the case where $\sabs{\Delta(\cP_1)} \leq \abs{\cP_1}^{5/4}$; this condition ensures that in both Theorems \ref{thm:expand-one-set} and \ref{thm:expand-two-sets}, the first term is smaller than the second. Then Theorems \ref{thm:expand-one-set} and \ref{thm:expand-two-sets} imply the following bounds. 

\begin{corollary}\label{cor:distances-one-set}
    Let $\ell_1$ and $\ell_2$ be two lines which are not parallel (but may be orthogonal), and let $\cP_1 \subseteq \ell_1$ and $\cP_2 \subseteq \ell_2$. If $\sabs{\Delta(\cP_1)} \leq \abs{\cP_1}^{5/4}$, then for all $\eps > 0$ we have \[\abs{\Delta(\cP_1, \cP_2)} = \Omega\left(\min\left\{\frac{\abs{\cP_1}^{28/13 - \eps}\abs{\cP_2}^{4/13 - \eps}}{\abs{\Delta(\cP_1)}^{12/13}}, \abs{\cP_1}^2, \abs{\cP_1}\abs{\cP_2}\right\}\right).\] 
\end{corollary}

\begin{corollary}\label{cor:distances-two-sets}
    Let $\ell_1$ and $\ell_2$ be two lines which are not parallel, and let $\cP_1 \subseteq \ell_1$ and $\cP_2 \subseteq \ell_2$. If $\sabs{\Delta(\cP_1)} \leq \abs{\cP_1}^{5/4}$ and $\abs{\Delta(\cP_2)} \leq \abs{\cP_2}^{4/3}$, then for all $\eps > 0$ we have \[\abs{\Delta(\cP_1, \cP_2)} = \Omega\left(\min\left\{\frac{\abs{\cP_1}^{256/121 - \eps}\abs{\cP_2}^{74/121 - \eps}}{\abs{\Delta(\cP_1)}^{108/121}\abs{\Delta(\cP_2)}^{24/121}}, \abs{\cP_1}^2, \abs{\cP_1}\abs{\cP_2}\right\}\right).\] 
\end{corollary}

\subsubsection{Heavy lines in configurations with few distances}

After Guth and Katz \cite{GK15} almost fully resolved the question of the minimum number of distinct distances $n$ points can form, it became natural to try to characterize the point sets that achieve this minimum. In fact, Erd\H{o}s \cite{Erd86} asked this question decades earlier --- specifically, he asked whether any such point set must have `lattice structure.' As a first step towards such a statement, he asked whether any such point set must have a line containing $\Omega(n^{1/2})$ points, or even just $\Omega(n^\eps)$ points. A variant of a proof of Szemer\'edi (communicated in \cite{Erd75}) shows that there must be a line containing $\Omega(\sqrt{\log n})$ points, and \cite{She14} shows that we can improve this bound to $\Omega(\log n)$. 

Meanwhile, we have substantially better bounds in the opposite direction, showing that there cannot be \emph{too many} points on a single line. Sheffer, Zahl, and de Zeeuw \cite{SZdZ16} proved that in a set of $n$ points forming $o(n)$ distinct distances, no line can contain $\Omega(n^{7/8})$ of these points. Raz, Roche-Newton, and Sharir \cite{RRNS15} then improved this to the following bound. 

\begin{theorem}[Raz--Roche-Newton--Sharir]\label{thm:rrns15}
    Let $\cP$ be a set of $n$ points such that $\abs{\Delta(\cP)} \leq \frac{1}{5}n$. Then every line $\ell$ satisfies $\abs{\ell \cap \cP} = O^*(n^{43/52})$. 
\end{theorem}

Combining the ideas of \cite{SZdZ16} and \cite{RRNS15} with Corollary \ref{cor:distances-one-set}, we get the following very slight improvement; we present the proof in Subsection \ref{subsec:heavy}. 

\begin{corollary}\label{cor:heavy-lines}
    Let $\cP$ be a set of $n$ points such that $\abs{\Delta(\cP)} \leq \frac{1}{5}n$. Then for all $\eps > 0$, every line $\ell$ satisfies $\abs{\ell \cap \cP} = O(n^{13/16 + \eps})$. 
\end{corollary}

\subsection{Almost tight bounds}\label{subsec:almost-tight}

Interestingly, several authors have studied the expansion of polynomials on sets with \emph{multiplicative} structure, and this problem is much better understood. Chang \cite{Cha06} first studied this problem for the specific polynomial $f(x, y) = x + y$ and proved that when $AA$ is very small, this polynomial essentially has the maximum possible expansion. 

\begin{theorem}[Chang]\label{thm:cha06}
    Fix $\eps > 0$. Then for all sets $A \subseteq \RR$ with $\abs{AA} = o(\abs{A}\log \abs{A})$ (with the constant in the asymptotic notation depending on $\eps$) and all sets $B \subseteq \RR$, we have \[\abs{A + B} = \Omega(\abs{A}\abs{B}^{1 - \eps}).\] 
\end{theorem}

The second author \cite{Poh20} then studied this problem for general polynomials and showed that when $AA$ is very small, \emph{all} polynomials $f \in \RR[x, y]$ except those of a certain special form have the maximum possible expansion. 

\begin{theorem}[Pohoata]\label{thm:poh20}
    Let $f \in \RR[x, y]$ be a polynomial which is not of the form $g(p(x, y))$ for some single-variable polynomial $g$ and some monomial $p$, and let $K$ be a constant. Then for all sets $A$ with $\abs{AA} \leq K\abs{A}$, we have \[\abs{f(A, A)} = \Omega(\abs{A}^2).\] 
\end{theorem}

In contrast, Theorems \ref{thm:expand-one-set} and \ref{thm:expand-two-sets} only lead to maximum expansion when $A$ is much larger than $B$. In the special case where $f(x, y) = x^2 + y^2$ and $A = B$, we obtain a bound that does give nearly the maximum possible expansion when $A + A$ is minimal, using different methods.

\begin{prop}\label{prop:asquared}
    Let $A \subseteq \RR$ be a finite set with $\abs{A + A} \leq K\abs{A}$. Then we have \[\sabs{A^2 + A^2} = \Omega\left(\frac{\abs{A}^2}{K^{12} \log \abs{A}}\right).\]  
\end{prop}

\subsection{Overview of techniques}\label{subsec:overview}

We now give a high-level overview of the techniques we use to prove Theorems \ref{thm:expand-one-set} and \ref{thm:expand-two-sets}. 

Expanding polynomials are often studied using incidence geometry. In particular, Raz, Sharir, and Solymosi \cite{RSS16} proved Theorem \ref{thm:rss16} by studying the quantity
\begin{align}
    \abs{\{(a_1, a_2, b_1, b_2) \in A^2 \times B^2 \mid f(a_1, b_1) = f(a_2, b_2)\}},\label{eqn:energy}
\end{align} 
which we think of as the \emph{energy} associated to $f$. If $f(A, B)$ is small, then this energy must be large (by Cauchy--Schwarz); so to lower-bound $f(A, B)$, it suffices to upper-bound this energy. To do so, they set up a collection of points and curves whose incidences correspond to $4$-tuples of this form, and then used incidence bounds to upper-bound the number of such $4$-tuples.  

The way we use the information that $A + A$ is small is by using the fact that every $a \in A$ can be written as $\alpha - a'$ for $\alpha \in A + A$ and $a' \in A$ in $\abs{A}$ ways (if we choose $a'$, then $\alpha = a + a'$ is automatically in $A + A$). So instead of working with the energy \eqref{eqn:energy}, we work with \[\abs{\{(\alpha_1, \alpha_2, a_1', a_2', b_1, b_2) \in (A + A)^2 \times A^2 \times B^2 \mid f(\alpha_1 - a_1', b_1) = f(\alpha_2 - a_2', b_2)\}}.\] As before, if $f(A, B)$ is small, then this quantity must be large; and we can again get upper bounds on it using incidence bounds. When $A + A$ is small, our incidence problem will have a smaller set of points, leading to better bounds. 

The proofs of Theorems \ref{thm:expand-one-set} and \ref{thm:expand-two-sets} both follow this basic setup. In both cases, the same curve may repeat multiple times in our incidence problem. The main difference between the proofs is that in the proof of Theorem \ref{thm:expand-two-sets}, we use the additive structure of $B$ to get better control on curve multiplicities.

The above approach does not lead to good bounds; we amplify these bounds by using \emph{proximity}. Solymosi and Zahl \cite{SZ24} introduced the idea of proximity in order to prove Theorem \ref{thm:sz24}. Roughly speaking, instead of considering \emph{all} $4$-tuples satisfying $f(a_1, b_1) = f(a_2, b_2)$, they restricted their attention to ones where $a_1$ and $a_2$, as well as $b_1$ and $b_2$, are `close' to each other. These proximity conditions shrink our lower bound on the number of such $4$-tuples in terms of $\abs{f(A, B)}$, but they also shrink the number of points and curves in the incidence problem. It turns out that they shrink the upper bound on the number of incidences by a smaller factor than they shrink the lower bound, giving us a better bound on $\abs{f(A, B)}$. Intuitively, this is because among solutions to $f(a_1, b_1) = f(a_2, b_2)$, there is a correlation between $a_1$ and $a_2$ being close and $b_1$ and $b_2$ being close. Thus, imposing \emph{both} proximity conditions does not shrink the lower bound by much more than imposing \emph{one} would. 

Our proofs work by incorporating proximity into the ideas mentioned above --- instead of considering all $6$-tuples satisfying $f(\alpha_1 - a_1', b_1) = f(\alpha_2 - a_2', b_2)$, we only consider pairs $(\alpha_1, \alpha_2)$, $(a_1', a_2')$, and $(b_1, b_2)$ whose two elements are close. This shrinks our upper bound on the number of incidences by a smaller factor than the lower bound, resulting in better bounds for $\abs{f(A, B)}$.

The organization of this paper is as follows. In Section \ref{sec:prelims}, we state a few preliminaries needed for our proofs. In Section \ref{sec:warmup}, we present a simpler version of our proofs without proximity (obtaining weaker bounds), to illustrate some of the other main ideas --- in particular, how we set up the incidence problem and make use of the additive structure of $A$ and $B$. In Section \ref{sec:outline}, we give an outline of the actual proofs with proximity (for both Theorems \ref{thm:expand-one-set} and \ref{thm:expand-two-sets}), and in Sections \ref{sec:rt-lower}--\ref{sec:conclusion}, we fill in the details of this outline. Finally, in Subsection \ref{subsec:asquared}, we prove Proposition \ref{prop:asquared}, and in Subsection \ref{subsec:heavy}, we prove Corollary \ref{cor:heavy-lines}. 

\section{Preliminaries}\label{sec:prelims}

In this section, we state a few preliminaries that we will need for our proofs. We use standard asymptotic notation: 
\begin{itemize}
    \item $x = O(y)$ means there is a constant $c > 0$ such that $x \leq cy$;
    \item $x = \Omega(y)$ means there is a constant $c > 0$ such that $x \geq cy$;
    \item $x = \Theta(y)$ means we have both $x = O(y)$ and $x = \Omega(y)$. 
\end{itemize} 
We use $O^*(\blank)$ and $\Omega^*(\blank)$ to suppress logarithmic factors --- more precisely, $x = O^*(y)$ means there are constants $c, d > 0$ such that $x \leq cy(\log y)^d$, and $x = \Omega^*(y)$ means there are constants $c, d > 0$ such that $x \geq cy(\log y)^{-d}$. 

All logarithms are base $2$. 

\subsection{Incidence bounds}\label{subsec:sharir--zahl}

Our proofs rely on an incidence bound for algebraic curves from \cite{SZ17}; in this subsection, we give a few relevant definitions and state this bound. 

For a polynomial $f \in \RR[x, y]$, we write \[\cZ(f) = \{(x, y) \in \RR^2 \mid f(x, y) = 0\}.\] We say a set $\gamma \subseteq \RR^2$ is an \emph{algebraic curve} if $\gamma$ is nonempty and we can write $\gamma = \cZ(f)$ for some nonconstant polynomial $f \in \RR[x, y]$. 

We say a polynomial $f \in \RR[x, y]$ is \emph{reducible} if we can write $f = f_1f_2$ for nonconstant polynomials $f_1, f_2 \in \RR[x, y]$; we say $f$ is \emph{irreducible} if it is not reducible. We say an algebraic curve $\gamma$ is \emph{reducible} if we can write $\gamma = \cZ(f)$ for some irreducible $f \in \RR[x, y]$. 

We say a collection of curves $\cC$ is a \emph{$k$-dimensional family} if the curves in $\cC$ are of the form $\cZ(f)$ for polynomials $f \in \RR[x, y]$ whose coefficients are themselves polynomials in $k$ parameters $p_1$, \ldots, $p_k$. More formally, this means there is a polynomial $F \in \RR[x, y, p_1, \ldots, p_k]$ such that if we write \[f_{p_1, \ldots, p_k}(x, y) = F(x, y, p_1, \ldots, p_k)\] (where we think of $p_1, \ldots, p_k \in \RR$ as fixed, and $f_{p_1, \ldots, p_k}$ as a polynomial in $x$ and $y$), then we have \[\cC = \{\cZ(f_{p_1, \ldots, p_k}) \mid p_1, \ldots, p_k \in \RR\}.\] We define the \emph{degree} of $\cC$ as $\deg F$ (more precisely, the minimum value of $\deg F$ over all $F$ which could be used to define $\cC$). For example, the collection of circles in $\RR^2$ is a $3$-dimensional family of degree $2$, corresponding to \[F(x, y, z_1, z_2, z_3) = (x - z_1)^2 + (y - z_2)^2 - z_3^2.\]

For a collection of points $\Pi$ and a collection of curves $\Gamma$ in $\RR^2$, we say an \emph{incidence} between $\Pi$ and $\Gamma$ is a pair $(p, \gamma) \in \Pi \times \Gamma$ where $p$ lies on $\gamma$. We write \[I(\Pi, \Gamma) = \abs{\{(p, \gamma) \in \Pi \times \Gamma \mid p \in \gamma\}}\] to denote the number of incidences between $\Pi$ and $\Gamma$. 

We will then use the following incidence bound, due to Sharir and Zahl \cite{SZ17}. (Sharir and Zahl actually work in a somewhat more general setting, but the setting we have given here is easier to describe and is enough for our purposes.)

\begin{theorem}[Sharir--Zahl]\label{thm:sharir--zahl}
    Let $\Pi$ be a set of points in $\RR^2$, and $\Gamma$ a set of (distinct) irreducible algebraic curves in $\RR^2$ from a $k$-dimensional family of degree $d$. Then for all $\eps > 0$, we have \[I(\Pi, \Gamma) = O(\abs{\Pi}^{\frac{2k}{5k - 4}}\abs{\Gamma}^{\frac{5k - 6}{5k - 4} + \eps} + \abs{\Pi}^{2/3}\abs{\Gamma}^{2/3} + \abs{\Pi} + \abs{\Gamma})\] (where the implicit constant depends on $k$, $d$, and $\eps$).  
\end{theorem}

In our applications of Theorem \ref{thm:sharir--zahl}, we will always have $k = 3$, in which case this bound gives \[I(\Pi, \Gamma) = O(\abs{\Pi}^{6/11}\abs{\Gamma}^{9/11 + \eps} + \abs{\Pi}^{2/3}\abs{\Gamma}^{2/3} + \abs{\Pi} + \abs{\Gamma}).\] 

\subsection{Irreducibility of polynomials}

When applying Theorem \ref{thm:sharir--zahl}, we will need to ensure that certain polynomials are irreducible; we now state a result that we will use for this purpose. 

We say a polynomial $f \in \RR[x, y]$ is \emph{decomposable} if we can write $f(x, y) = q(r(x, y))$ for some $r \in \RR[x, y]$ and $q \in \RR[z]$ with $\deg q \geq 2$. We say $f$ is \emph{indecomposable} if it is not decomposable.

The following result is a combination of theorems of Stein \cite{Ste89}, who proved this result with $\RR$ replaced by $\CC$, and Ayad \cite{Aya02}, who proved that for polynomials in $\RR[x, y]$, decomposability over $\RR$ is equivalent to decomposability over $\CC$. (See \cite{RSS16} for more details.)

\begin{theorem}[Stein, Ayad]\label{thm:irred}
    If $f \in \RR[x, y]$ is indecomposable, then there are at most $\deg f$ values of $\lambda \in \RR$ for which $f(x, y) - \lambda$ is reducible. 
\end{theorem}

\subsection{Cutting curves into monotone pieces}\label{subsec:cut-monotone}

Finally, for the lower bounds involving proximity, we will need a statement that any algebraic curve can be split into a constant number of monotone pieces. To formalize this, we say a set $M \subseteq \RR^2$ is \emph{monotone} if either for all $(x_1, y_1)$ and $(x_2, y_2)$ in $M$ with $x_1 < x_2$, we have $y_1 \leq y_2$, or for all such $(x_1, y_1)$ and $(x_2, y_2)$, we have $y_1 \geq y_2$. 

\begin{fact}\label{fact:cut-monotone}
    For any constant-degree nonconstant polynomial $f \in \RR[x, y]$, we can partition $\cZ(f)$ into $O(1)$ monotone sets (where the constant depends on $\deg f$). 
\end{fact}

Fact \ref{fact:cut-monotone} can be proven by cutting $\cZ(f)$ at its intersections with $\cZ(\partial_x f)$ and $\cZ(\partial_y f)$; each piece (i.e., connected component) of the result is monotone, and we can use Bezout's theorem together with the argument in \cite[Section 2.2]{SZ17} to bound the number of pieces this creates. (This in fact gives a partition of $\cZ(f)$ into \emph{connected} monotone sets, but we will not need connectivity in our applications of Fact \ref{fact:cut-monotone}.)

\section{Warmup: a simpler argument without proximity}\label{sec:warmup}

In this section, we sketch a simpler version of the argument that does not use proximity, and ends up with the weaker bounds of 
\begin{equation}
    \abs{f(A, B)} = \Omega\left(\min\left\{\frac{\abs{A}^{26/11 - \eps}\abs{B}^{2/11 - \eps}}{\abs{A + A}^{12/11}}, \frac{\abs{A}^{8/3}\abs{B}^{1/3 - \eps}}{\abs{A + A}^{4/3}}, \abs{A}^2, \abs{A}\abs{B}\right\}\right)\label{eqn:one-set-weak-bound}
\end{equation} 
in place of Theorem \ref{thm:expand-one-set} and 
\begin{equation}
    \abs{f(A, B)} = \Omega\left(\min\left\{\frac{\abs{A}^{26/11 - \eps}\abs{B}^{52/99 - \eps}}{\abs{A + A}^{12/11}\abs{B + B}^{8/33}}, \frac{\abs{A}^{8/3}\abs{B}^{26/27 - \eps}}{\abs{A + A}^{4/3}\abs{B + B}^{4/9}}, \abs{A}^2, \abs{A}\abs{B}\right\}\right)\label{eqn:two-sets-weak-bound}
\end{equation}
in place of Theorem \ref{thm:expand-two-sets}. For simplicity, in this section we assume that $\abs{A} \approx \abs{B}$ and that $A + A$ is not much bigger than $A$ (then the first terms in both bounds are the smallest ones, so these are the bounds we will end up with). Later in the current proof, when we apply Theorem \ref{thm:sharir--zahl}, this assumption will make the first term dominate the bound of the theorem. 

For brevity, we set $d = \deg f$. We also use $\eta$ to denote a small error parameter which we will eventually take to be sufficiently small with respect to $\eps$; its purpose is essentially to keep track of all `subpolynomial' factors. The implicit constants in our bounds may depend on $d$ and $\eta$. 

First, instead of directly working with $f(A, B)$, we work with the set \[\cQ = \{(a_1, a_2, b_1, b_2) \in A^2 \times B^2 \mid f(a_1, b_1) = f(a_2, b_2)\}\] (we think of $\abs{\cQ}$ as the \emph{energy} associated to $f$). By Cauchy--Schwarz, we have 
\begin{equation}
    \abs{\cQ} = \sum_\delta \abs{\{(a, b) \mid f(a, b) = \delta\}}^2 \geq \frac{(\sum_\delta \abs{\{(a, b) \mid f(a, b) = \delta\}})^2}{\abs{f(A, B)}} = \frac{\abs{A}^2\abs{B}^2}{\abs{f(A, B)}}.\label{eqn:cauchy--schwarz}
\end{equation} 
So to get a lower bound on $\abs{f(A, B)}$, it suffices to get an \emph{upper} bound on $\abs{\cQ}$. 

With the irreducibility condition of Theorem \ref{thm:sharir--zahl} in mind, we say a pair $(b_1, b_2) \in B^2$ is \emph{bad} if either $g(x) - g(y) + h(b_1) - h(b_2)$ or $h(x) - h(y) + h(b_1) - h(b_2)$ is reducible, and we define \[\cQ' = \{(a_1, a_2, b_1, b_2) \in \cQ \mid \text{$(b_1, b_2)$ not bad}\}.\] We can show that bad pairs do not contribute too much to $\cQ$ --- specifically, we have \[\sabs{\cQ \setminus \cQ'} = O(\abs{A}\abs{B}).\] (We will later state this as Lemma \ref{lem:remove-bad}, which we prove in Subsection \ref{subsec:remove-bad}.) If $\abs{\cQ} \leq 2\abs{\cQ \setminus \cQ'}$, then this means $\abs{\cQ} = O(\abs{A}\abs{B})$, and \eqref{eqn:cauchy--schwarz} implies that $\abs{f(A, B)} = \Omega(\abs{A}\abs{B})$. So we can assume $\abs{\cQ} \geq 2\abs{\cQ \setminus \cQ'}$, which means $\abs{\cQ'} = \Omega(\abs{\cQ})$. 

To use the fact that $A + A$ is small, we consider the set 
\begin{multline*}
    \cR = \{(\alpha_1, \alpha_2, a_1', a_2', b_1, b_2) \in (A + A)^2 \times A^2 \times B^2 \mid \\
    f(\alpha_1 - a_1', b_1) = f(\alpha_2 - a_2', b_2), \text{$(b_1, b_2)$ not bad}\}.
\end{multline*}
On one hand, we have $\sabs{\cR} \geq \abs{A}^2\sabs{\cQ'} = \Omega(\abs{A}^2\sabs{\cQ})$ --- given any $(a_1, a_2, b_1, b_2) \in \cQ'$, we can choose $a_1', a_2' \in A$ arbitrarily and set $\alpha_1 = a_1 + a_1'$ and $\alpha_2 = a_2 + a_2'$ to get an element of $\cR'$.

On the other hand, we can get upper bounds for $\abs{\cR}$ by setting up an incidence problem with points representing $(\alpha_1, \alpha_2)$ and curves representing $(a_1', a_2', b_1, b_2)$, such that incidences between these points and curves correspond to elements of $\cR$, and applying Theorem \ref{thm:sharir--zahl}. More precisely, we define a set of points \[\Pi = (A + A)^2.\] We define a \emph{multiset} of curves $\Gamma$ where for each $(a_1', a_2', b_1, b_2) \in A^2 \times B^2$ such that $(b_1, b_2)$ is not bad, we include the curve defined by $f(x_1 - a_1', b_1) = f(x_2 - a_2', b_2)$, or more explicitly 
\begin{equation}
    g(x_1 - a_1' + p(b_1)) + h(b_1) = g(x_2 - a_2' + p(b_2)) + h(b_2). \label{eqn:curve-eqn}
\end{equation}  
Then we have $\abs{\Pi} = \abs{A + A}^2$ and $\abs{\Gamma} \leq \abs{A}^2\abs{B}^2$, and $\sabs{\cR} = I(\Pi, \Gamma)$ by definition. 

The curves in $\Gamma$ are defined by the three parameters $a_1' - p(b_1)$, $a_2' - p(b_2)$, and $h(b_1) - h(b_2)$; this means they belong to a $3$-dimensional family, as defined in Subsection \ref{subsec:sharir--zahl}. Also, all curves in $\Gamma$ are irreducible, since if the curve defined by \eqref{eqn:curve-eqn} were reducible, then $(b_1, b_2)$ would be bad. So we are in a setting where we can apply Theorem \ref{thm:sharir--zahl} to bound $I(\Pi, \Gamma)$. 

However, Theorem \ref{thm:sharir--zahl} can only be applied to a set of \emph{distinct} curves, while our collection $\Gamma$ may contain repeated curves. To handle this, we first dyadically partition $\Gamma$ based on curve multiplicity --- more precisely, for each $m \in \{2^0, 2^1, \ldots\}$, we let $\Gamma_m$ be the portion of $\Gamma$ consisting of curves with multiplicities in $[m, 2m)$ (where we include curves in $\Gamma_m$ with the same multiplicities as in $\Gamma$). Then the multisets $\Gamma_m$ partition $\Gamma$, so \[I(\Pi, \Gamma) = \sum_m I(\Pi, \Gamma_m).\] Then for each $m$, we let $\Gamma_m'$ be the variant of $\Gamma_m$ where each curve appears only once, so that $I(\Pi, \Gamma_m) = \Theta(m \cdot I(\Pi, \Gamma_m'))$. We can now use Theorem \ref{thm:sharir--zahl} to bound $I(\Pi, \Gamma_m')$, as $\Gamma_m'$ is a set of \emph{distinct} irreducible algebraic curves from a $3$-dimensional family; this gives 
\begin{equation}
    I(\Pi, \Gamma_m) = O(m\abs{\Pi}^{6/11}\sabs{\Gamma_m'}^{9/11 + \eta} + m\abs{\Pi}^{2/3}\sabs{\Gamma_m'}^{2/3} + m\abs{\Pi} + m\sabs{\Gamma_m'}).\label{eqn:inc-bound-dyadic} 
\end{equation}
For simplicity, right now we only consider the first term of this bound. (When we sum over $m$, the contribution of the first term ends up dominating when $\abs{A} \approx \abs{B}$ and $A + A$ is small, which is the regime we are focusing on.) 

We will get an upper bound on $m\sabs{\Gamma_m'}^{9/11}$ by bounding each of $m\sabs{\Gamma_m'}$ and $m^2\sabs{\Gamma_m'}$. For the former, we have $m\sabs{\Gamma_m'} \leq \sabs{\Gamma}$, as each curve in $\Gamma_m'$ appears at least $m$ times in $\Gamma$. For the latter, $m^2\sabs{\Gamma_m'}$ is at most the number of \emph{pairs} of $4$-tuples $(a_1', a_2', b_1, b_2)$ and $(a_3', a_4', b_3, b_4)$ producing the same curve. Given the equation of a curve of the form \eqref{eqn:curve-eqn}, the coefficient of the second-highest power of $x_1$ uniquely determines $a_1' - p(b_1)$; similarly, the coefficient of the second-highest power of $x_2$ determines $a_2' - p(b_2)$. Finally, once we know $a_1' - p(b_1)$ and $a_2' - p(b_2)$, the constant term uniquely determines $h(b_1) - h(b_2)$. This means two $4$-tuples producing the same curve must satisfy 
\begin{align*}
    a_1' - p(b_1) &= a_3' - p(b_3), & a_2' - p(b_2) &= a_4' - p(b_4), & h(b_1) - h(b_2) &= h(b_3) - h(b_4).
\end{align*}
So if we define the set
\begin{equation*}
    \cS = \{(b_1, b_2, c_1, c_2) \in B^2 \times B^2 \mid h(b_1) + h(c_1) = h(b_2) + h(c_2), \, \text{$(b_1, b_2)$, $(c_1, c_2)$ not bad}\},
\end{equation*}
then we have $m^2\sabs{\Gamma_m'} \leq \abs{A}^2\sabs{\cS}$.

Now we have upper bounds on both $m\sabs{\Gamma_m'}$ and $m^2\sabs{\Gamma_m'}$, and we can combine them to get a bound on $m\sabs{\Gamma_m'}^{9/11}$ by writing \[m\sabs{\Gamma_m'}^{9/11} = (m\sabs{\Gamma_m'})^{7/11}(m^2\sabs{\Gamma_m'})^{2/11} \leq \abs{\Gamma}^{7/11}\abs{A}^{4/11}\sabs{\cS}^{2/11}.\] Then for each $m$, the first term of \eqref{eqn:inc-bound-dyadic} satisfies \[m\abs{\Pi}^{6/11}\sabs{\Gamma_m'}^{9/11 + \eta} \leq \abs{A}^{18/11 + 2\eta}\abs{B}^{14/11 + 2\eta}\abs{A + A}^{12/11}\sabs{\cS}^{2/11}\] (after plugging in $\abs{\Pi} = \Theta(\abs{A + A}^2)$ and $\abs{\Gamma} = O(\abs{A}^2\abs{B}^2)$). 

Finally, our dyadic partition has at most $O(\log \abs{B}) = O(\abs{B}^\eta)$ values of $m$, since the argument we used to bound pairs of identical curves also shows that all curve multiplicities are $O(\abs{B})$. So summing this bound over all chunks of the partition, we get \[\sabs{\cR} = \sum_m I(\Pi, \Gamma_m) = O(\abs{A}^{18/11 + 2\eta}\abs{B}^{14/11 + 3\eta}\abs{A + A}^{12/11}\sabs{\cS}^{2/11}).\] 

Now combining this with the lower bound $\sabs{\cR} = \Omega(\abs{A}^2\sabs{\cQ})$ gives 
\begin{equation}
    \sabs{\cQ} = O(\abs{A}^{-4/11 + 2\eta}\abs{B}^{14/11 + 3\eta}\abs{A + A}^{12/11}\sabs{\cS}^{2/11}).\label{eqn:q-to-qh}
\end{equation}
Our goal was originally to get an upper bound on $\abs{\cQ}$, and \eqref{eqn:q-to-qh} bounds $\abs{\cQ}$ in terms of $\abs{\cS}$, so now we want an upper bound on $\abs{\cS}$. In the setting of Theorem \ref{thm:expand-one-set}, where we do not have any information about the additive structure of $B$, we simply use the crude bound \[\sabs{\cS} = O(\abs{B}^3)\] (this bound comes from the fact that once we choose $b_1$, $b_2$, and $c_1$, there are $O(1)$ choices for $c_2$ such that $h(b_1) + h(c_1) = h(b_2) + h(c_2)$). This ends up giving \eqref{eqn:one-set-weak-bound} (or rather, its first term), our weaker version of Theorem \ref{thm:expand-one-set}. 

Meanwhile, in the setting of Theorem \ref{thm:expand-two-sets}, where we \emph{do} have information about the additive structure of $B$, we can do better than this crude bound. The point is that $\cS$ is almost exactly the analog of $\cQ$ if we were working with the polynomial $h(x) + h(y)$ on $B \times B$, rather than with $f(x, y) = g(x + p(y)) + h(y)$ on $A \times B$ (except that we exclude bad pairs from $\cS$). And so far our argument has given an upper bound on $\abs{\cQ}$ in terms of $\abs{\cS}$, so if we run it \emph{again} for this new setting, then we will end up with a similar upper bound on $\abs{\cS}$ in terms of $\abs{\cS}$ itself. 

So we define the set 
\begin{multline*}
    \cT = \{(\beta_1, \beta_2, b_1', b_2', c_1, c_2) \in (B + B)^2 \times B^2 \times B^2 \mid \\
    h(\beta_1 - b_1') + h(c_1) = h(\beta_2 - b_2') + h(c_2), \text{$(c_1, c_2)$ not bad}\},
\end{multline*}
which is the analog of $\cR$ for this new setting. Then we have $\sabs{\cT} \geq \abs{B}^2\sabs{\cS}$, while we can get an upper bound on $\abs{\cT}$ by setting up an incidence problem in the same way as before. Combining these bounds gives \[\sabs{\cS} = O(\abs{B}^{10/11 + 5\eta}\abs{B + B}^{12/11}\sabs{\cS}^{2/11})\] (this is the same bound as in \eqref{eqn:q-to-qh}, but with $A$ replaced with $B$ and $\cQ$ with $\cS$). And now we can cancel out $\sabs{\cS}^{2/11}$ from the right-hand side to get
\begin{equation}
    \abs{\cS} = O(\abs{B}^{10/9 + 7\eta}\abs{B + B}^{4/3}),\label{eqn:s-bound-final}
\end{equation}
and plug this upper bound into \eqref{eqn:q-to-qh} to get an upper bound on $\abs{\cQ}$. This ends up giving \eqref{eqn:two-sets-weak-bound} (or rather, its first term), our weaker version of Theorem \ref{thm:expand-two-sets}. 

\section{Outline of the full proof}\label{sec:outline}

In this section, we give an outline of the proofs of Theorem \ref{thm:expand-one-set} and \ref{thm:expand-two-sets}, in which we state the key steps as lemmas; we will prove these lemmas in the following sections. 

As in Section \ref{sec:warmup}, we set $d = \deg f$, and we let $\eta$ denote a small error parameter meant to keep track of subpolynomial factors, which we will eventually take to be sufficiently small relative to $\eps$. The implicit constants in our asymptotic notation may depend on $d$ and $\eta$, but not anything else. 

As mentioned in Subsection \ref{subsec:overview}, the way we improve the weaker bounds in Section \ref{sec:warmup} is by using proximity. Roughly speaking, instead of working with $\cR$ itself, we work with a `proximity-restricted' variant where we require the two elements of each of the pairs $(\alpha_1, \alpha_2)$, $(a_1', a_2')$, and $(b_1, b_2)$ to be `close' to each other. 

To state such a proximity condition on pairs from a set $X$, we write $X = \{x_1, x_2, \ldots, x_n\}$ with $x_1 < x_2 < \cdots < x_n$. For $x \in X$, we let $i_X(x)$ denote the index of $x$ in $X$ (i.e., the index $i$ for which $x = x_i$). For example, if $X = \{1, 2, 5, 6, 7\}$, then $i_X(6) = 4$. For a parameter $t \in [0, 1]$, we say a pair $(x_1, x_2) \in X^2$ is \emph{$t$-close} if \[\abs{i_X(x_1) - i_X(x_2)} \leq t\abs{X}.\] If $t\abs{X} < 1$, then this definition is silly; but if $t\abs{X} \geq 1$, then a $\Theta(t)$-fraction of pairs $(x_1, x_2) \in X^2$ are $t$-close. Throughout the argument, we leave $t$ as a parameter and assume that $t\abs{A}, t\abs{B} \geq 1$. When we set a value of $t$ in Subsection \ref{subsec:set-t}, we will verify that it satisfies these inequalities. 

Then the proximity-restricted variant of $\cR$ we work with is
\begin{multline*}
    \cR_t = \{(\alpha_1, \alpha_2, a_1', a_2', b_1, b_2) \in (A + A)^2 \times A^2 \times B^2 \mid f(\alpha_1 - a_1', b_1) = f(\alpha_2 - a_2', b_2), \\
    \text{$(\alpha_1, \alpha_2)$, $(a_1', a_2')$, $(b_1, b_2)$ $t$-close, $(b_1, b_2)$ not bad}\}.
\end{multline*}
Explicitly, these proximity conditions mean that 
\begin{align*}
    \abs{i_{A + A}(\alpha_1) - i_{A + A}(\alpha_2)} &\leq t\abs{A + A}, &
    \abs{i_A(a_1') - i_A(a_2')} &\leq t\abs{A}, &
    \abs{i_B(b_1) - i_B(b_2)} &\leq t\abs{B}.
\end{align*}
For example, when we say $(\alpha_1, \alpha_2)$ is $t$-close, the ambient set is $A + A$. 

In Section \ref{sec:warmup}, we had the lower bound $\sabs{\cR} = \Omega(\abs{A}^2\abs{\cQ})$. The following lemma shows that the proximity conditions do not shrink this lower bound by too much.

\begin{lemma}\label{lem:rt-lower}
    We have $\sabs{\cR_t} = \Omega(t^2\abs{A}^2\sabs{\cQ}) - O(t\abs{A}^3\abs{B})$. 
\end{lemma}

More explicitly, this notation means that there are constants $c, c' > 0$ such that 
\begin{equation}
    \sabs{\cR_t} \geq c \cdot t^2\abs{A}^2\sabs{\cQ} - c' \cdot t\abs{A}^3\abs{B}.\label{eqn:rt-lower-explicit}
\end{equation} 
The second term in Lemma \ref{lem:rt-lower} comes from the removal of bad pairs, and we will eventually set $t$ large enough to make this term negligible. When focusing on the first term, this bound captures the intuition that our proximity conditions are `correlated' --- even though we are imposing \emph{three} proximity conditions and each shrinks a set of pairs by a factor of $t$, we only lose \emph{two} factors of $t$ in the lower bound. (This is, very loosely, the intuition behind why proximity leads to a better bound.)

We next set up an incidence problem corresponding to $\cR_t$ and upper-bound $\sabs{\cR_t}$ using incidence bounds, in the same way as in Section \ref{sec:warmup}. There, curve multiplicities in our incidence problem for $\cR$ were controlled by $\cS$. Here they will be controlled by the set
\begin{multline}
    \cS_t = \{(b_1, b_2, c_1, c_2) \in B^2 \times B^2 \mid h(b_1) + h(c_1) = h(b_2) + h(c_2), \\
    \text{$(b_1, b_2)$, $(c_1, c_2)$ $t$-close and not bad}\},\label{eqn:st-defin}
\end{multline}
which is a proximity-restricted variant of $\cS$. Then this incidence argument leads to the following upper bound on $\sabs{\cR_t}$. 

\begin{lemma}\label{lem:rt-upper}
    We have $\sabs{\cR_t} = O(\term{I} + \term{II} + \term{III} + \term{IV})$, where 
    \begin{align*}
        \term{I} &= t^2\abs{A}^{18/11 + 2\eta}\abs{B}^{14/11 + 3\eta}\abs{A + A}^{12/11}\sabs{\cS_t}^{2/11}, &
        \term{III} &= t\abs{A + A}^2\abs{B}, \\
        \term{II} &= t^{5/3}\abs{A}^{4/3}\abs{B}^{2/3 + \eta}\abs{A + A}^{4/3}\sabs{\cS_t}^{1/3}, &
        \term{IV} &= t^2\abs{A}^2\abs{B}^2.
    \end{align*}
\end{lemma}

To finish the argument, we need an upper bound on $\abs{\cS_t}$. For the proof of Theorem \ref{thm:expand-one-set}, we use a crude bound similar to the one in Section \ref{sec:warmup} --- we have \[\sabs{\cS_t} = O(t\abs{B}^3).\] This is because if we choose $b_1$ and $c_1$, then there are $\Theta(t\abs{B})$ choices for $b_2$ such that $(b_1, b_2)$ is $t$-close, and $O(1)$ choices for $c_2$ such that $h(b_1) + h(c_1) = h(b_2) + h(c_2)$. If we plug this bound on $\abs{\cS_t}$ into Lemma \ref{lem:rt-upper} and combine it with Lemma \ref{lem:rt-lower}, and set \[t = \Theta\left(\frac{\abs{A}\abs{B}}{\abs{\cQ}}\right),\] we end up with Theorem \ref{thm:expand-one-set}. (The reason for this choice is that it is the smallest we can make $t$ while ensuring that the subtraction in Lemma \ref{lem:rt-lower} is negligible.)

For the proof of Theorem \ref{thm:expand-two-sets}, we will use the additive structure of $B$ to get a better bound on $\abs{\cS_t}$. Unlike in Section \ref{sec:warmup}, we cannot do so by directly reusing the above argument (where we used the additive structure of $A$ to get an upper bound on $\sabs{\cQ}$ in terms of $\abs{\cS_t}$), because it would not maintain the proximity restriction. The issue is that if we defined a proximity-restricted variant of $\cT$ in the same way as $\cR_t$ (our proximity-restricted variant of $\cR$), the analog of Lemma \ref{lem:rt-lower} would only give a lower bound on its size in terms of $\abs{\cS}$, not $\abs{\cS_t}$. 

We overcome the above issue by defining a more complicated proximity-restricted variant of $\cT$ whose size we \emph{can} lower-bound in terms of $\abs{\cS_t}$ (and can again upper-bound using incidences). Combining these bounds leads to the following upper bound on $\abs{\cS_t}$. 

\begin{lemma}\label{lem:qth-upper}
    Assuming that $\deg h \geq 2$ and $\abs{B + B} \leq \abs{B}^{4/3}$, we have \[\sabs{\cS_t} = O(t^{11/9}\abs{B}^{10/9 + 8\eta}\abs{B + B}^{4/3}).\]
\end{lemma}

Up to subpolynomial factors, the bound on $\abs{\cS_t}$ in Lemma \ref{lem:qth-upper} matches \eqref{eqn:s-bound-final} (our upper bound on $\abs{\cS}$ from Section \ref{sec:warmup}), with an additional factor of $t^{11/9}$. Finally, plugging this bound on $\abs{\cS_t}$ into Lemma \ref{lem:rt-upper} and combining it with Lemma \ref{lem:rt-lower}, and setting the value of $t$, gives Theorem \ref{thm:expand-two-sets}. 

We prove Lemma \ref{lem:rt-lower} in Section \ref{sec:rt-lower}, Lemma \ref{lem:rt-upper} in Section \ref{sec:rt-upper}, and Lemma \ref{lem:qth-upper} in Section \ref{sec:qth-upper}; in Section \ref{sec:conclusion}, we put these lemmas together to deduce Theorems \ref{thm:expand-one-set} and \ref{thm:expand-two-sets}.  

\section{A lower bound on \texorpdfstring{$\cR_t$}{Rt}}\label{sec:rt-lower}

In this section, we prove Lemma \ref{lem:rt-lower}, which states that adding the proximity conditions to $\cR$ does not decrease the lower bound on its size by too much. The proof consists of two fairly disjoint parts. We will use the following lemma to handle the proximity conditions. (For the definition of a monotone set in $\RR^2$, see Subsection \ref{subsec:cut-monotone}.)

\begin{lemma}\label{lem:prox-lower-bound}
    Let $M \subseteq A \times B$ be monotone. Then there are $\Omega(t^2\abs{A}^2\abs{M}^2)$ $6$-tuples \[(a_1, a_2, a_1', a_2', b_1, b_2) \in A^2 \times A^2 \times B^2\] with $(a_1, b_1), (a_2, b_2) \in M$ such that all the pairs $(a_1, a_2) \in A^2$, $(a_1', a_2') \in A^2$, $(a_1 + a_1', a_2 + a_2') \in (A + A)^2$, and $(b_1, b_2) \in B^2$ are $t$-close. 
\end{lemma}

When ignoring the proximity conditions, there are exactly $\abs{A}^2\abs{M}^2$ such $6$-tuples; so Lemma \ref{lem:prox-lower-bound} shows that the proximity conditions only cost a factor of $t^2$. 

We will use the following lemma to handle the exclusion of bad pairs. 

\begin{lemma}\label{lem:remove-bad}
    We have $\abs{\{(a_1, a_2, b_1, b_2) \in \cQ \mid \text{$(b_1, b_2)$ bad}\}} = O(\abs{A}\abs{B})$. 
\end{lemma}

We prove Lemma \ref{lem:prox-lower-bound} in Subsection \ref{subsec:prox-lower-bound} and Lemma \ref{lem:remove-bad} in Subsection \ref{subsec:remove-bad}. In Subsection \ref{subsec:rt-lower-conclusion} we put them together to deduce Lemma \ref{lem:rt-lower}. 

\subsection{Handling the proximity conditions}\label{subsec:prox-lower-bound}

In this subsection, we prove Lemma \ref{lem:prox-lower-bound}. To do so, we first define a function $\varphi \colon A \to \RR$ as \[\varphi(a) = \frac{\sum_{a' \in A} i_{A + A}(a + a')}{\abs{A}}.\] Since $\varphi(a)$ is the average of various indices of elements of $A + A$, it takes values in $[1, \abs{A + A}]$. Also, $\varphi$ is increasing, since increasing $a$ only increases each term $a + a'$ in this sum (and therefore its index in $A + A$). 

Let the elements of $A$ be $x_1 < \cdots < x_{\abs{A}}$, and let the pairs in $M$ be $(w_1, z_1)$, \ldots, $(w_{\abs{M}}, z_{\abs{M}})$, sorted so that $w_1 \leq \cdots \leq w_{\abs{M}}$ and either $z_1 \leq \cdots \leq z_{\abs{M}}$ or $z_1 \geq \cdots \geq z_{\abs{M}}$ (we can sort $M$ in this way because $M$ is monotone). Finally, let $k = \sfloor{\frac{1}{32}t\abs{M}}$ and $\ell = \sfloor{\frac{1}{8}t\abs{A}}$. 

\begin{claim}\label{claim:choose-pairs-from-m}
    There are at least $\frac{1}{2}\abs{M}$ indices $1 \leq i \leq \abs{M} - k$ for which we have 
    \begin{align}
        i_A(y_{i + k}) - i_A(y_i) &\leq t\abs{A}, & 
        \abs{i_B(z_i) - i_B(z_{i + k})} &\leq t\abs{B}, &
        \varphi(y_{i + k}) - \varphi(y_i)& \leq \frac{t\abs{A + A}}{8}.\label{eqn:choose-pairs-from-m}
    \end{align}
\end{claim}

\begin{proof}
    For each of these three conditions, we will use an averaging argument to show that only a small fraction of indices $i$ \emph{fail} to satisfy the condition. For the first condition, we have \[\sum_{i = 1}^{\abs{M} - k} (i_A(y_{i + k}) - i_A(y_i)) \leq k\abs{A} \leq \frac{t\abs{M}\abs{A}}{32},\] since this sum telescopes and leaves $k$ positive and $k$ negative terms, each of which is at most $\abs{A}$. Since each term in the sum is nonnegative, at most $\frac{1}{32}\abs{M}$ of its terms can exceed $t\abs{A}$. 

    The same reasoning shows that at most $\frac{1}{32}\abs{M}$ indices fail to satisfy the second condition, and that at most $\frac{1}{4}\abs{M}$ indices fail to satisfy the third. So we are left with at least \[\abs{M} - k - \frac{1}{32}\abs{M} - \frac{1}{32}\abs{M} - \frac{1}{4}\abs{M} \geq \frac{1}{2}\abs{M}\] indices which satisfy all three conditions. 
\end{proof}

\begin{claim}\label{claim:phi-to-aprimes}
    Suppose that $a_1, a_2 \in A$ are such that $a_1 \leq a_2$ and $\varphi(a_2) - \varphi(a_1) \leq \frac{1}{8}t\abs{A + A}$. Then there are at least $\frac{1}{2}\abs{A}$ indices $1 \leq i \leq \abs{A} - \ell$ such that 
    \begin{equation}
        i_{A + A}(a_2 + x_{i + \ell}) - i_{A + A}(a_1 + x_i) \leq t\abs{A + A}.\label{eqn:phi-to-aprimes}
    \end{equation}
\end{claim}

\begin{proof}
    Similarly to the proof of Claim \ref{claim:choose-pairs-from-m}, we use an averaging argument to show that only a small fraction of indices $i$ fail to satisfy this condition. For this, we consider \[\sum_{i = 1}^{\abs{A} - \ell} (i_{A + A}(a_2 + x_{i + \ell}) - i_{A + A}(a_1 + x_i)),\] which we can split as 
    \begin{equation}
        \sum_{i = 1}^{\abs{A} - \ell} (i_{A + A}(a_2 + x_{i + \ell}) - i_{A + A}(a_2 + x_i)) + \sum_{i = 1}^{\abs{A} - \ell} (i_{A + A}(a_2 + x_i) - i_{A + A}(a_1 + x_i)).\label{eqn:split-prox-avg}
    \end{equation} 
    The first sum telescopes and leaves $\ell$ positive and negative terms, each between $1$ and $\abs{A + A}$, so \[\sum_{i = 1}^{\abs{A} - \ell} (i_{A + A}(a_2 + x_{i + \ell}) - i_{A + A}(a_2 + x_i)) \leq \ell \abs{A + A} \leq \frac{t\abs{A}\abs{A + A}}{8}.\] For the second term, by assumption we have \[\sum_{i = 1}^{\abs{A}} (i_{A + A}(a_2 + x_i) - i_{A + A}(a_1 + x_i)) = \abs{A}(\varphi(a_2) - \varphi(a_1)) \leq \frac{t\abs{A}\abs{A + A}}{8}.\] The second term of \eqref{eqn:split-prox-avg} is missing the last few terms of this sum, but the absence of these terms can only make it smaller, so it satisfies the same upper bound. This means \[\sum_{i = 1}^{\abs{A} - \ell} (i_{A + A}(a_2 + x_{i + \ell}) - i_{A + A}(a_1 + x_i)) \leq \frac{t\abs{A}\abs{A + A}}{4},\] and since all terms in this sum are nonnegative, at most $\frac{1}{4}\abs{A}$ of them can exceed $t\abs{A + A}$. This means at most $\frac{1}{4}\abs{A}$ indices $i$ fail to satisfy our condition, so at least \[\abs{A} - \ell - \frac{1}{4}\abs{A} \geq \frac{1}{2}\abs{A}\] indices do satisfy it.
\end{proof}

\begin{proof}[Proof of Lemma \ref{lem:prox-lower-bound}]
    To produce a $6$-tuple $(a_1, a_2, a_1', a_2', b_1, b_2)$ as in Lemma \ref{lem:prox-lower-bound}, we first choose an index $i$ satisfying \eqref{eqn:choose-pairs-from-m}, which can be done in at least $\frac{1}{2}\abs{M}$ ways by Claim \ref{claim:choose-pairs-from-m}, and an index $i \leq j \leq i + k$, which can be done in $k + 1 \geq \frac{1}{32}t\abs{M}$ ways; and we set \[(a_1, b_1) = (y_i, z_i) \quad \text{and} \quad (a_2, b_2) = (y_j, z_j).\] Then $a_1 \leq a_2$, and \eqref{eqn:choose-pairs-from-m} means that $(a_1, a_2)$ and $(b_1, b_2)$ are $t$-close and \[\varphi(a_2) - \varphi(a_1) \leq \frac{t\abs{A + A}}{8}.\] (Here we are using the fact that if \eqref{eqn:choose-pairs-from-m} holds, then it still holds if we replace the index $i + k$ with any $i \leq j \leq i + k$.)

    Now we choose an index $i'$ satisfying \eqref{eqn:phi-to-aprimes}, which can be done in at least $\frac{1}{2}\abs{A}$ ways by Claim \ref{claim:phi-to-aprimes}, and an index $i' \leq j' \leq i' + \ell$, which can be done in $\ell + 1 \geq \frac{1}{8}t\abs{A}$ ways; and we set \[a_1' = x_{i'} \quad \text{and} \quad a_2' = x_{j'}.\] Since $\abs{i_A(a_1') - i_A(a_2')} \leq \ell \leq t\abs{A}$, we have that $(a_1', a_2')$ is $t$-close, and \eqref{eqn:phi-to-aprimes} means that $(a_1 + a_1', a_2 + a_2')$ is $t$-close as well (again, if \eqref{eqn:phi-to-aprimes} holds, then it still holds if we replace $i + k$ with any $i \leq j \leq i + k$). So our $6$-tuple satisfies all the conditions of Lemma \ref{lem:prox-lower-bound}.

    The above process shows that the number of $6$-tuples $(a_1, a_2, a_1', a_2', b_1, b_2)$ satisfying the conditions of Lemma \ref{lem:prox-lower-bound} is at least \[\frac{1}{2}\abs{M} \cdot \frac{1}{32}t\abs{M} \cdot \frac{1}{2}\abs{A} \cdot \frac{1}{8}t\abs{A} = \frac{1}{1024}t^2\abs{A}^2\abs{M}^2.\qedhere\]
\end{proof} 

\subsection{Handling bad pairs}\label{subsec:remove-bad}

In this subsection, we prove Lemma \ref{lem:remove-bad}, our upper bound on the contribution of bad pairs to $\cQ$. Recall that $(b_1, b_2) \in B^2$ is bad if and only if either $g(x) - g(y) + h(b_1) - h(b_2)$ or $h(x) - h(y) + h(b_1) - h(b_2)$ is reducible. We will first use Theorem \ref{thm:irred} to bound the \emph{number} of bad pairs; to apply this theorem, we need the following observation. 

\begin{claim}\label{claim:indec}
    For every nonconstant $p \in \RR[x]$, the polynomial $p(x) - p(y)$ is indecomposable. 
\end{claim}

\begin{proof}
    Assume for contradiction that $p(x) - p(y)$ is decomposable, so we can write \[p(x) - p(y) = q(r(x, y))\] for some polynomials $r \in \RR[x, y]$ and $q \in \RR[z]$ with $\deg q \geq 2$. Let the leading coefficients of $p$ and $q$ be $c_p$ and $c_q$, and let $r_*$ be the sum of the monomials in $r$ of degree exactly $\deg r$. Then the sum of the maximal-degree monomials in $q(r(x, y))$ is $c_qr_*(x, y)^{\deg q}$, while the sum of the maximal-degree monomials in $p(x) - p(y)$ is $c_p(x^{\deg p} - y^{\deg p})$, so we have \[c_p(x^{\deg p} - y^{\deg p}) = c_qr_*(x, y)^{\deg q}.\] Since the two sides must have the same degree, we have $\deg p = (\deg q)(\deg r)$; then letting $z = \frac{x}{y}$ and dividing by $y^{\deg p}$ gives \[c_p(z^{\deg p} - 1) = c_qr_*(z, 1)^{\deg q}.\] In particular, this means some multiple of $z^{\deg p} - 1$ is the $(\deg q)$th power of some polynomial (namely $r_*(z, 1)$). But this is impossible, as its complex roots all have multiplicities $1$, rather than multiplicities divisible by $\deg q$. 
\end{proof}

\begin{proof}[Proof of Lemma \ref{lem:remove-bad}]
    Claim \ref{claim:indec} allows us to apply Theorem \ref{thm:irred} to $g(x) - g(y)$ and $h(x) - h(y)$. Then by Theorem \ref{thm:irred}, at most $\deg g + \deg h \leq 2d$ values $\lambda \in \RR$ satisfy that $g(x) - g(y) + \lambda$ or $h(x) - h(y) + \lambda$ is reducible, so there are at most $2d$ possible values of $h(b_1) - h(b_2)$ among bad pairs $(b_1, b_2)$. Furthermore, given the values of $h(b_1) - h(b_2)$ and $b_1$, there are at most $d$ ways to choose $b_2$. So there are at most $2d \cdot \abs{B} \cdot d = O(\abs{B})$ bad pairs. 

    Finally, once we have chosen a bad pair $(b_1, b_2)$, in order to have $(a_1, a_2, b_1, b_2) \in \cQ$ we must have $f(a_1, b_1) = f(a_2, b_2)$, i.e., \[g(a_1 + p(b_1)) + h(b_1) = g(a_2 + p(b_2)) + h(b_2).\] This means after choosing $a_1$, there are at most $d = O(1)$ choices for $a_2$. 

    So in total, the number of $4$-tuples in $\cQ$ for which $(b_1, b_2)$ is bad is $O(\abs{B} \cdot \abs{A} \cdot 1) = O(\abs{A}\abs{B})$.
\end{proof}

\subsection{Proof of Lemma \ref{lem:rt-lower}}\label{subsec:rt-lower-conclusion}

We first define
\begin{multline*}
    \widetilde{\cR_t} = \{(a_1, a_2, a_1', a_2', b_1, b_2) \in A^2 \times A^2 \times B^2 \mid f(a_1, b_1) = f(a_2, b_2), \\
    \text{$(a_1, a_2)$, $(a_1', a_2')$, $(a_1 + a_1', a_2 + a_2')$, $(b_1, b_2)$ $t$-close}\}.
\end{multline*}
First, we can use Lemma \ref{lem:prox-lower-bound} to prove a lower bound on $\sabs{\widetilde{\cR_t}}$. 

\begin{claim}\label{claim:st-lower}
    We have $\sabs{\widetilde{\cR_t}} = \Omega(t^2\abs{A}^2\abs{\cQ})$. 
\end{claim}

\begin{proof}
    We will consider each $\delta \in f(A, B)$ separately and show that the number of elements of $\widetilde{\cR_t}$ where $f(a_1, b_1) = f(a_2, b_2) = \delta$ (which we refer to as the contribution of $\delta$ to $\widetilde{\cR_t}$) is $\Omega(t^2\abs{A}^2)$ times the number of elements of $\cQ$ with $f(a_1, b_1) = f(a_2, b_2) = \delta$ (which we refer to as its contribution to $\cQ$). To do so, first define \[K_\delta = \{(a, b) \in A \times B \mid f(a, b) = \delta\};\] then the contribution of $\delta$ to $\cQ$ is $\abs{K_\delta}^2$.

    Meanwhile, by Fact \ref{fact:cut-monotone}, we can partition the curve defined by $f(x, y) = \delta$ into $O(1)$ monotone sets. One of these sets must contain at least a constant fraction of $K_\delta$; let $M_\delta$ be the subset of $K_\delta$ lying in this set, so that $\abs{M_\delta} = \Theta(\abs{K_\delta})$ and $M_\delta$ is monotone. 

    Then applying Lemma \ref{lem:prox-lower-bound} to $M_\delta$ gives that the contribution of $\delta$ to $\widetilde{\cR_t}$ is \[\Omega(t^2\abs{A}^2\abs{M_\delta}^2) = \Omega(t^2\abs{A}^2\abs{K_\delta}^2),\] which is $\Omega(t^2\abs{A}^2)$ times the contribution of $\delta$ to $\cQ$. 
    
    Finally, summing over all $\delta \in f(A, B)$ gives that $\sabs{\widetilde{\cR_t}} = \Omega(t^2\abs{A}^2\abs{\cQ})$. 
\end{proof}

Now given any $(a_1, a_2, a_1', a_2', b_1, b_2) \in \widetilde{\cR_t}$ where $(b_1, b_2)$ is not bad, we can produce an element of $\cR_t$ by setting $\alpha_1 = a_1 + a_1'$ and $\alpha_2 = a_2 + a_2'$. To bound the number of elements of $\widetilde{\cR_t}$ where $(b_1, b_2)$ \emph{is} bad, Lemma \ref{lem:remove-bad} gives that there are $O(\abs{A}\abs{B})$ choices for $(a_1, a_2, b_1, b_2)$ for which $(b_1, b_2)$ is bad, while there are $\Theta(t\abs{A}^2)$ choices for $(a_1', a_2')$, so there are \[O(\abs{A}\abs{B} \cdot t\abs{A}^2) = O(t\abs{A}^3\abs{B})\] such elements of $\widetilde{\cR_t}$. Then the number of elements of $\widetilde{\cR_t}$ where $(b_1, b_2)$ is \emph{not} bad is at least \[\sabs{\widetilde{\cR_t}} - O(t\abs{A}^3\abs{B}) = \Omega(t^2\abs{A}^2\abs{\cQ}) - O(t\abs{A}^3\abs{B}),\] and each such element of $\widetilde{\cR_t}$ produces an element of $\cR_t$; this finishes the proof of Lemma \ref{lem:rt-lower}. 

\section{An incidence problem and upper bound on \texorpdfstring{$\cR_t$}{Rt}}\label{sec:rt-upper}

In this section, we prove Lemma \ref{lem:rt-upper}, our upper bound on $\sabs{\cR_t}$. We will do so by setting up a collection of points $\Pi$ and curves $\Gamma$ in $\RR^2$ whose incidences correspond to elements of $\cR_t$, and then using Theorem \ref{thm:sharir--zahl} to bound the number of incidences.

As in Section \ref{sec:warmup}, we create a point for each pair $(\alpha_1, \alpha_2)$ and a curve for each $4$-tuple $(a_1', a_2', b_1, b_2)$. We define the set of points \[\Pi = \{(\alpha_1, \alpha_2) \in (A + A)^2 \mid \text{$(\alpha_1, \alpha_2)$ $t$-close}\}.\] We define a multiset $\Gamma$ of curves as follows. For each $(a_1', a_2', b_1, b_2) \in A^2 \times B^2$ such that $(a_1', a_2')$ and $(b_1, b_2)$ are $t$-close and $(b_1, b_2)$ is not bad, we include the curve defined by $f(x_1 - a_1', b_1) = f(x_2 - a_2', b_2)$, or more explicitly \[g(x_1 - a_1' + p(b_1)) + h(b_1) = g(x_2 - a_2' + p(b_2)) + h(b_2).\] Then $\abs{\Pi} = \Theta(t\abs{A + A}^2)$ and $\abs{\Gamma} = O(t^2\abs{A}^2\abs{B}^2)$; and by definition, we have $\sabs{\cR_t} = I(\Pi, \Gamma)$.

As in Section \ref{sec:warmup}, the curves of $\Gamma$ belong to a $3$-dimensional family with parameters $a_1' - p(b_1)$, $a_2' - p(b_2)$, and $h(b_1) - h(b_2)$. Since we exclude bad pairs $(b_1, b_2)$, all curves of $\Gamma$ are irreducible. So to be able to apply Theorem \ref{thm:sharir--zahl}, it remains to control curve multiplicities. 

As seen in Section \ref{sec:warmup}, the equation of a curve in $\Gamma$ uniquely determines the values of $a_1' - p(b_1)$, $a_2' - p(b_2)$, and $h(b_1) - h(b_2)$. So to study curve multiplicities, we consider the number of ways to choose $a_1'$, $a_2'$, $b_1$, and $b_2$ with the three above parameters fixed. First, there are at most $\abs{B}$ choices for $b_1$, and then at most $d = O(1)$ choices for $b_2$; then $a_1'$ and $a_2'$ are uniquely determined. So the multiplicity of each curve is $O(\abs{B})$. 

We now dyadically partition $\Gamma$ based on curve multiplicities, as in Section \ref{sec:warmup}: For each $m \in \{2^0, 2^1, \ldots\}$, we let $\Gamma_m$ be the portion of $\Gamma$ consisting of curves with multiplicities in $[m, 2m)$. We include curves with their original multiplicities in $\Gamma$, so that $I(\Pi, \Gamma) = \sum_m I(\Pi, \Gamma_m)$. Then for each $m$, we let $\Gamma_m'$ be the variant of $\Gamma_m$ where we include curves only once, so that $I(\Pi, \Gamma_m) = \Theta(m \cdot I(\Pi, \Gamma_m'))$. Then Theorem \ref{thm:sharir--zahl} applied to $\Pi$ and $\Gamma_m'$ implies that
\begin{equation}
    I(\Pi, \Gamma_m) = O(m\abs{\Pi}^{6/11}\sabs{\Gamma_m'}^{9/11 + \eta} + m\abs{\Pi}^{2/3}\sabs{\Gamma_m'}^{2/3} + m\abs{\Pi} + m\sabs{\Gamma_m'}).\label{eqn:incidence-bound-rt}
\end{equation}
To obtain an upper bound on $I(\Pi, \Gamma)$, we sum \eqref{eqn:incidence-bound-rt} over $m$; we handle the contribution of each term of \eqref{eqn:incidence-bound-rt} to this sum separately. 

For the fourth term, we have $m\sabs{\Gamma_m'} \leq \abs{\Gamma_m}$ for each $m$, and $\sum_m \abs{\Gamma_m} = \abs{\Gamma}$. So the total contribution of this term is $O(\abs{\Gamma}) = O(t^2\abs{A}^2\abs{B}^2)$, corresponding to the term $\term{IV}$ in Lemma \ref{lem:rt-upper}. 

When summing the third term $m\abs{\Pi}$ over $m$, we obtain a geometric progression that sums to \[O(\abs{B}\abs{\Pi}) = O(t\abs{A + A}^2\abs{B}),\] corresponding to the term $\term{III}$. 

For the first and second terms of \eqref{eqn:incidence-bound-rt}, we will combine the fact that \[m\sabs{\Gamma_m'} \leq \sabs{\Gamma} = O(t^2\abs{A}^2\abs{B}^2)\] with the following bound on $m^2\sabs{\Gamma_m'}$ in terms of $\abs{\cS_t}$ (see \eqref{eqn:st-defin} for the definition of $\cS_t$). 

\begin{claim}\label{claim:mult-to-qth}
    For each $m$, we have $m^2\sabs{\Gamma_m'} = O(t\abs{A}^2\sabs{\cS_t})$. 
\end{claim}

\begin{proof}
    We count pairs of $4$-tuples $(a_1', a_2', b_1, b_2)$ and $(a_3', a_4', b_3, b_4)$ producing the same curve. On one hand, there are at least $m^2\sabs{\Gamma_m'}$ such pairs, since each curve in $\Gamma_m'$ can be produced in at least $m$ ways. On the other hand, if $(a_1', a_2', b_1, b_2)$ and $(a_3', a_4', b_3, b_4)$ produce the same curve, then we must have 
    \begin{align*}
        a_1' - p(b_1) &= a_3' - p(b_3), & a_2' - p(b_2) &= a_4' - p(b_4), & h(b_1) - h(b_2) &= h(b_3) - h(b_4).
    \end{align*}
    In particular, we must have $(b_1, b_2, b_4, b_3) \in \cS_t$ --- the third equation rearranges to \[h(b_1) + h(b_4) = h(b_2) + h(b_3),\] and $(b_1, b_2)$ and $(b_4, b_3)$ are both $t$-close and not bad (we are given that $(b_3, b_4)$ is $t$-close and not bad, but both conditions are symmetric, so the same is true of $(b_4, b_3)$). 

    So there are $\sabs{\cS_t}$ choices for $b_1$, $b_2$, $b_3$, and $b_4$; then there are $\Theta(t\abs{A}^2)$ choices for $a_1'$ and $a_2'$ (since $(a_1', a_2')$ must be $t$-close); and then $a_3'$ and $a_4'$ are uniquely determined. This means there are $O(t\abs{A}^2\sabs{\cS_t})$ such pairs of $4$-tuples. Combining this with the lower bound that there are at least $m^2\sabs{\Gamma_m'}$ such pairs leads to the assertion of the claim.  
\end{proof}

Then to deal with the first term of \eqref{eqn:incidence-bound-rt}, we can write \[m\sabs{\Gamma_m'}^{9/11} = (m\sabs{\Gamma_m'})^{7/11}(m^2\sabs{\Gamma_m'})^{2/11} = O(t^{16/11}\abs{A}^{18/11}\abs{B}^{14/11}\sabs{\cS_t}^{2/11}).\] Plugging in $\sabs{\Pi} = \Theta(t\abs{A + A}^2)$ and summing over all $O(\log \abs{B}) = O(\abs{B}^\eta)$ chunks of the dyadic partition (i.e., values of $m$), we get that the total contribution of the first term is at most \[O(t^2\abs{A}^{18/11 + 2\eta}\abs{B}^{14/11 + 3\eta}\abs{A + A}^{12/11}\sabs{\cS_t}^{2/11}),\] corresponding to the term $\term{I}$ of Lemma \ref{lem:rt-upper}. 

Similarly, to deal with the second term, we can write \[m\sabs{\Gamma_m'}^{2/3} = (m\sabs{\Gamma_m'})^{1/3}(m^2\sabs{\Gamma_m'})^{1/3} = O(t\abs{A}^{4/3}\abs{B}^{2/3}\sabs{\cS_t}^{1/3}),\] so the total contribution of the second term is at most \[O(t^{5/3}\abs{A}^{4/3}\abs{B}^{2/3 + \eta}\abs{A + A}^{4/3}\sabs{\cS_t}^{1/3}),\] corresponding to the term $\term{II}$. 

\section{An upper bound on \texorpdfstring{$\cS_t$}{qth}}\label{sec:qth-upper}

We now prove Lemma \ref{lem:qth-upper}, our upper bound on $\abs{\cS_t}$ when $B$ has additive structure. We assume in this section that $\deg h \geq 2$ and $\abs{B + B} \leq \abs{B}^{4/3}$, as in the hypothesis of Lemma \ref{lem:qth-upper}. 

\subsection{Overview of the proof}

We first give an overview of the proof, where we explain the main ideas and state the key steps as lemmas; we prove these lemmas in the following subsections. 

The motivation behind the proof is that $\cS_t$ is essentially an analog of $\cQ$ for the polynomial $h(x) + h(y)$ on $B \times B$ (in place of $f(x, y)$ on $A \times B$), but with additional proximity conditions --- recall that we defined \[\cQ = \{(a_1, a_2, b_1, b_2) \in A^2 \times B^2 \mid f(a_1, b_1) = f(a_2, b_2)\},\] while we defined
\begin{multline*}
    \cS_t = \{(b_1, b_2, c_1, c_2) \in B^2 \times B^2 \mid h(b_1) + h(c_1) = h(b_2) + h(c_2), \\
    \text{$(b_1, b_2)$, $(c_1, c_2)$ $t$-close and not bad}\}.
\end{multline*}
So far (i.e., in Sections \ref{sec:rt-lower} and \ref{sec:rt-upper}), we have proven an upper bound on $\abs{\cQ}$ in terms of $\abs{\cS_t}$ by defining  
\begin{multline*}
    \cR_t = \{(\alpha_1, \alpha_2, a_1', a_2', b_1, b_2) \in (A + A)^2 \times A^2 \times B^2 \mid f(\alpha_1 - a_1', b_1) = f(\alpha_2 - a_2', b_2), \\
    \text{$(\alpha_1, \alpha_2)$, $(a_1', a_2')$, $(b_1, b_2)$ $t$-close, $(b_1, b_2)$ not bad}\},
\end{multline*} 
proving a lower bound on $\abs{\cR_t}$ in terms of $\abs{\cQ}$ (Lemma \ref{lem:rt-lower}), and proving an upper bound on $\abs{\cR_t}$ in terms of $\abs{\cS_t}$ using incidence bounds (Lemma \ref{lem:rt-upper}). We will use a similar approach to prove an upper bound on $\abs{\cS_t}$ in terms of $\abs{\cS_t}$ itself --- we will define a set somewhat analogously to $\cR_t$ for the polynomial $h(x) + h(y)$, prove a lower bound on its size in terms of $\abs{\cS_t}$, and prove an upper bound on its size using incidence bounds.

However, we cannot define this set \emph{completely} analogously to $\cR_t$, because unlike in Lemma \ref{lem:rt-lower} (which gave a lower bound on $\abs{\cR_t}$ in terms of $\abs{\cQ}$), here we need to incorporate the proximity conditions in $\cS_t$ into our lower bound. So we instead define an analog of $\cR_t$ with `looser' proximity restrictions, so that we \emph{are} able to get a good lower bound on its size in terms of $\abs{\cS_t}$. To do so, we first define \[P = \{(b_1, b_2) \in B^2 \mid \text{$(b_1, b_2)$ $t$-close and not bad}\},\] so that $\cS_t$ is a subset of $P \times P$. We first need to find a subset $R \times S \subseteq P \times P$ such that each pair $(b_1, b_2) \in R$ corresponds to roughly the same number of pairs $(c_1, c_2) \in S$ --- we will define our analog of $\cR_t$ by only considering the portion of $\cS_t$ that lies in $R \times S$, and we will use this property to prove the desired lower bound. Intuitively, $R$ being small will weaken our final bound on $\sabs{\cS_t}$ by forcing us to use looser proximity conditions, while $S$ being small will strengthen our bound by allowing us to use fewer curves in the incidence problem. The following lemma (whose proof crucially relies on the symmetry of our setting) allows us to ensure that these sets are the \emph{same} size, so that these considerations balance each other out. 

\begin{lemma}\label{lem:r-and-s}
    There exist sets $R, S \subseteq P$ such that the following statements hold:
    \begin{itemize}
        \item We have $\sabs{\cS_t \cap (R \times S)} = \Omega(\abs{B}^{-\eta}\sabs{\cS_t})$. 
        \item There is some $m$ such that for each pair $(b_1, b_2) \in R$, there are between $m$ and $2m$ pairs $(c_1, c_2) \in S$ for which $(b_1, b_2, c_1, c_2) \in \cS_t$. 
        \item We have $\abs{R} = \abs{S}$. 
    \end{itemize}
\end{lemma}

Then we define $\cS_t^* = \cS_t \cap (R \times S)$, and we let \[r = \max\left\{\frac{\abs{R}}{64t\abs{B}^2}, t\right\} \quad \text{and} \quad s = \frac{\abs{S}}{64t\abs{B}^2}.\] Note that $r, s \in (0, 1)$ and $r \geq s$. The analog of $\cR_t$ we work with is the set
\begin{multline*}
    \cT_{t, r} = \{(\beta_1, \beta_2, b_1', b_2', c_1, c_2) \in (B + B)^2 \times B^2 \times S \mid \\
    h(\beta_1 + b_1') + h(c_1) = h(\beta_2 + b_2') + h(c_2), \, \text{$(\beta_1, \beta_2)$, $(b_1', b_2')$ $\tfrac{t}{r}$-close}\}.
\end{multline*}
Compared to the direct analog of $\cR_t$, we have weakened the proximity conditions on $(\beta_1, \beta_2)$ and $(b_1', b_2')$ based on how small $R$ is. We also draw $(c_1, c_2)$ only from $S$, rather than just requiring it to be $t$-close and not bad (which would correspond to drawing it from $P$). 

We will first prove the following lower bound on $\sabs{\cT_{t, r}}$ in terms of $\sabs{\cS_t^*}$. 

\begin{lemma}\label{lem:rtr-lower}
    We have $\sabs{\cT_{t, r}} = \Omega(\frac{t}{r}\abs{B}^2\sabs{\cS_t^*})$. 
\end{lemma}

We will then use incidence bounds to prove the following upper bound on $\sabs{\cT_{t, r}}$. 

\begin{lemma}\label{lem:rtr-upper}
    We have $\sabs{\cT_{t, r}} = O(\term{I} + \term{II} + \term{III} + \term{IV})$, where
    \begin{align*}
        \term{I} &= \frac{s^{7/11}t^2}{r^{15/11}}\abs{B}^{32/11 + 5\eta}\abs{B + B}^{12/11}\sabs{\cS_t}^{2/11}, &
        \term{III} &= \frac{t}{r}\abs{B}\abs{B + B}^2, \\[4pt]
        \term{II} &= \frac{s^{1/3}t^{5/3}}{r^{4/3}}\abs{B}^{2 + \eta}\abs{B + B}^{4/3}\sabs{\cS_t}^{1/3}, &
        \term{IV} &= \frac{st^2}{r}\abs{B}^4.
    \end{align*}
\end{lemma}

We prove Lemma \ref{lem:r-and-s} in Subsection \ref{subsec:r-and-s}, Lemma \ref{lem:rtr-lower} in Subsection \ref{subsec:rtr-lower}, and Lemma \ref{lem:rtr-upper} in Subsection \ref{subsec:rtr-upper}; in Subsection \ref{subsec:qth-upper-conclusion}, we combine these bounds to deduce Lemma \ref{lem:qth-upper}. 

\subsection{Restricting the host sets}\label{subsec:r-and-s}

In this subsection, we prove Lemma \ref{lem:r-and-s}. To do so, we rewrite the definition of $\cS_t$ as \[\cS_t = \{(b_1, b_2, c_1, c_2) \in P \times P \mid h(b_1) - h(b_2) = h(c_2) - h(c_1)\}.\] We then split $\cS_t$ based on the common value of $h(b_1) - h(b_2)$ and $h(c_2) - h(c_1)$, which we denote by $\delta$. More precisely, for each $\delta \in h(B) - h(B)$, we define \[P_\delta = \{(b_1, b_2) \in P \mid h(b_1) - h(b_2) = \delta\},\] so that $\cS_t = \bigcup_\delta (P_\delta \times P_{-\delta})$. Note that $P$ is symmetric, meaning that $(b_2, b_1)$ is in $P$ if and only if $(b_1, b_2)$ is. This means $\abs{P_\delta} = \abs{P_{-\delta}}$ for every $\delta$, so 
\begin{equation}
    \sabs{\cS_t} = \sum_\delta \abs{P_\delta}^2.\label{eqn:st-sum-pdelta}
\end{equation}

We now dyadically partition the values of $\delta$ based on $\abs{P_\delta}$ --- for each $m \in \{2^0, 2^1, \ldots\}$, we define \[\Delta_m = \{\delta \in h(B) - h(B) \mid \abs{P_\delta} \in [m, 2m)\}.\] Since we have $\abs{P_\delta} \leq \abs{B}^2$ for all $\delta$, there are $O(\log \abs{B}) = O(\abs{B}^\eta)$ parts in this dyadic partition, so there must be some $m$ which accounts for an $\Omega(\abs{B}^{-\eta})$-fraction of \eqref{eqn:st-sum-pdelta}. In other words, \[\sum_{\delta \in \Delta_m} \abs{P_\delta}^2 = \Omega(\abs{B}^{-\eta}\sabs{\cS_t}).\] Finally, we fix such a value of $m$ and set $R = \bigcup_{\delta \in \Delta_m} P_\delta$ and $S = \bigcup_{\delta \in \Delta_m} P_{-\delta}$. These sets have the properties described in Lemma \ref{lem:r-and-s} --- we have $\cS_t \cap (R \times S) = \bigcup_{\delta \in \Delta_m} (P_\delta \times P_{-\delta})$, and the fact that $\sabs{P_\delta} = \sabs{P_{-\delta}}$ for all $\delta$ implies that $\abs{R} = \abs{S}$. 

\subsection{A lower bound on \texorpdfstring{$\cT_{t, r}$}{Rtrh}}\label{subsec:rtr-lower}

In this subsection, we prove Lemma \ref{lem:rtr-lower}, our lower bound on $\sabs{\cT_{t, r}}$. First, to provide some intuition for Lemma \ref{lem:rtr-lower}, we can try to produce a $6$-tuple in $\cT_{t, r}$ by starting with some $4$-tuple $(b_1, b_2, c_1, c_2) \in \cS_t^*$, choosing a $\frac{t}{r}$-close pair $(b_1', b_2') \in B^2$, and setting $\beta_1 = b_1 + b_1'$ and $\beta_2 = b_2 + b_2'$. There are \[\Theta\left(\frac{t}{r}\abs{B}^2\sabs{\cS_t^*}\right)\] ways to make these choices, and the resulting $6$-tuple $(\beta_1, \beta_2, b_1', b_2', c_1, c_2)$ satisfies all the conditions in the definition of $\cT_{t, r}$, except possibly the proximity condition on $(\beta_1, \beta_2)$. Our proof of Lemma \ref{lem:rtr-lower} will show that a positive proportion of $6$-tuples produced in this way do satisfy the proximity conditions on $(\beta_1, \beta_2)$.

The definition of $r$ implies that $r \geq t$. If $r = t$, then the desired statement is immediate, as the proximity condition on $(\beta_1, \beta_2)$ is vacuous. So it remains to consider the case where $r > t$, in which $\abs{R} = 64rt\abs{B}^2$. 

For simplicity, we assume without loss of generality that at least half of the pairs $(b_1, b_2) \in R$ satisfy $b_1 \leq b_2$ (otherwise we can swap the roles of the indices $1$ and $2$). Let $R'$ be the subset of $R$ consisting of these pairs, so that $\abs{R'} \geq 32rt\abs{B}^2$. 

We will handle the proximity condition on $(\beta_1, \beta_2)$ in a similar way to how we handled the proximity condition on $(\alpha_1, \alpha_2)$ in Section \ref{sec:rt-lower}. We define $\psi \colon B \to \RR$ as \[\psi(b) = \frac{\sum_{b' \in B} i_{B + B}(b + b')}{\abs{B}},\] so that $\psi$ is increasing and takes values in $[1, \abs{B + B}]$. 

Let the elements of $B$ be $y_1 < \cdots < y_{\abs{B}}$, and let $k = \sfloor{t\abs{B}}$ and $\ell = \sfloor{\frac{t}{8r}\abs{B}}$. 

\begin{claim}\label{claim:psi-skip}
    There are at most $4r\abs{B}$ indices $1 \leq i \leq \abs{B} - k$ for which 
    \begin{equation}
        \psi(y_{i + k}) - \psi(y_i) > \frac{t\abs{B + B}}{4r}.\label{eqn:psi-skip}
    \end{equation}
\end{claim}

\begin{proof}
    We use an averaging argument similar to the one from Claim \ref{claim:choose-pairs-from-m}. We have \[\sum_{i = 1}^{\abs{B} - k} (\psi(y_{i + k}) - \psi(y_i)) \leq k\abs{B + B} \leq t\abs{B}\abs{B + B},\] since this sum telescopes and leaves $k$ positive and negative terms, each at most $\abs{B + B}$. All terms in this sum are nonnegative, so at most $4r\abs{B}$ of them can be greater than $\frac{t}{4r}\abs{B + B}$. 
\end{proof}

\begin{claim}\label{claim:psi-to-bprime}
    Suppose that $b_1 \leq b_2$ and $\psi(b_2) - \psi(b_1) \leq \frac{t}{4r}\abs{B + B}$. Then for at least $\frac{1}{2}\abs{B}$ indices $1 \leq i \leq \abs{B} - \ell$, we have 
    \begin{equation}
        i_{B + B}(b_2 + y_{i + \ell}) - i_{B + B}(b_1 + y_i) \leq \frac{t\abs{B + B}}{r}.\label{eqn:psi-to-bprime}
    \end{equation}
\end{claim}

\begin{proof}
    Similarly to the proof of Claim \ref{claim:phi-to-aprimes}, we sum $i_{B + B}(b_2 + y_{i + \ell}) - i_{B + B}(b_1 + y_i)$ over all $1 \leq i \leq \abs{B} - \ell$, and split this sum into \[\sum_{i = 1}^{\abs{B} - \ell} (i_{B + B}(b_2 + y_{i + \ell}) - i_{B + B}(b_2 + y_i)) + \sum_{i = 1}^{\abs{B} - \ell} (i_{B + B}(b_2 + y_i) - i_{B + B}(b_1 + y_i)).\] The first sum telescopes, so it is at most $\ell\abs{B + B} \leq \frac{t}{8r}\abs{B}\abs{B + B}$, and the second sum is at most $\abs{B}(\psi(b_2) - \psi(b_1)) \leq \frac{t}{4r}\abs{B}\abs{B + B}$. Combining these gives that \[\sum_{i = 1}^{\abs{B} - \ell} (i_{B + B}(b_2 + y_{i + \ell}) - i_{B + B}(b_1 + y_i)) \leq \frac{3t\abs{B}\abs{B + B}}{8r}.\] Finally, since all terms in this sum are nonnegative, at most $\frac{3}{8}\abs{B}$ of them can be greater than $\frac{t}{r}\abs{B + B}$. So the number of indices $i$ satisfying the given inequality is at least \[\abs{B} - \ell - \frac{3}{8}\abs{B} \geq \frac{1}{2}\abs{B}.\qedhere\] 
\end{proof}

\begin{proof}[Proof of Lemma \ref{lem:rtr-lower}]
    First, we say that an element $y_i \in B$ is \emph{friendly} if $1 \leq i \leq \abs{B} - k$ and $i$ satisfies \eqref{eqn:psi-skip}. Otherwise, we say $y_i$ is \emph{unfriendly}. By Claim \ref{claim:psi-skip} and the fact that $t \leq r$, the number of unfriendly elements of $B$ is at most \[4r\abs{B} + k \leq 4r\abs{B} + t\abs{B} \leq 8r\abs{B}.\] This means the number of pairs $(b_1, b_2) \in R'$ in which $b_1$ is unfriendly is at most \[8r\abs{B} \cdot 2t\abs{B} = 16rt\abs{B}^2,\] since given $b_1$, there are at most $t\abs{B} + 1 \leq 2t\abs{B}$ choices for $b_2$ such that $b_1 \leq b_2$ and $(b_1, b_2)$ is $t$-close. Then since $\abs{R'} \geq 32rt\abs{B}^2$, there are at least $16rt\abs{B}^2 = \frac{1}{4}\abs{R}$ pairs $(b_1, b_2) \in R'$ for which $b_1$ is friendly. By definition, any such pair satisfies $\psi(b_2) - \psi(b_1) \leq \frac{t}{4r}\abs{B + B}$. 

    Then we can construct $6$-tuples $(\beta_1, \beta_2, b_1', b_2', c_1, c_2)$ in $\cT_{t, r}$ as follows. We first choose $(b_1, b_2) \in R'$ such that $b_1$ is friendly, which can be done in at least $\frac{1}{4}\abs{R}$ ways. Then we choose an index $i$ satisfying \eqref{eqn:psi-to-bprime} (which can be done in at least $\frac{1}{2}\abs{B}$ ways by Claim \ref{claim:psi-to-bprime}) and an index $i \leq j \leq i + \ell$ (which can be done in $\ell + 1 \geq \frac{t}{8r}\abs{B}$ ways), and set $b_1' = y_i$ and $b_2' = y_j$ (this immediately means $(b_1', b_2')$ is $\frac{t}{r}$-close). We then set $\beta_1 = b_1 + b_1'$ and $\beta_2 = b_2 + b_2'$; then \eqref{eqn:psi-to-bprime} means that $(\beta_1, \beta_2)$ is $\frac{t}{r}$-close. Finally, we choose some $(c_1, c_2) \in S$ such that $(b_1, b_2, c_1, c_2)$ is in $\cS_t^*$; this can be done in at least $m$ ways (where $m$ is as in Lemma \ref{lem:r-and-s}). So this gives at least \[\frac{1}{4}\abs{R} \cdot \frac{1}{2}\abs{B} \cdot \frac{t}{8r}\abs{B} \cdot m = \frac{t}{64r}\abs{B}^2\cdot m\abs{R}\] choices. Finally, we have $2m\abs{R} \geq \sabs{\cS_t^*}$ (since each $(b_1, b_2) \in R$ corresponds to at most $2m$ choices of $(c_1, c_2) \in S$ for which $(b_1, b_2, c_1, c_2) \in \cS_t^*$), giving that \[\sabs{\cT_{t, r}} \geq \frac{t}{128r}\abs{B}^2\sabs{\cS_t^*}.\qedhere\]
\end{proof}

\subsection{An incidence problem and upper bound on \texorpdfstring{$\cT_{t, r}$}{Rtrh}}\label{subsec:rtr-upper}

In this subsection, we prove Lemma \ref{lem:rtr-upper}. As in Section \ref{sec:rt-upper}, we set up a collection of points $\Pi$ and curves $\Gamma$ whose incidences correspond to elements of $\cT_{t, r}$, and use Theorem \ref{thm:sharir--zahl} to bound their number of incidences. 

First, we define \[\Pi = \{(\beta_1, \beta_2) \in (B + B)^2 \mid \text{$(\beta_1, \beta_2)$ $\tfrac{t}{r}$-close}\}.\] We then define $\Gamma$ to be the multiset of curves where for each $(b_1', b_2', c_1, c_2) \in B^2 \times S$ such that $(b_1', b_2')$ is $\frac{t}{r}$-close, we include the curve \[h(x_1 - b_1') + h(c_1) = h(x_2 - b_2') + h(c_2).\] Then $\abs{\Pi} = \Theta(\frac{t}{r}\abs{B + B}^2)$ and $\abs{\Gamma} = O(\frac{t}{r}\abs{B}^2\abs{S}) = O(\frac{st^2}{r}\abs{B}^4)$, and by definition, $\sabs{\cT_{t, r}} = I(\Pi, \Gamma)$.

As in Section \ref{sec:rt-upper}, the curves in $\Gamma$ are irreducible and belong to a $3$-dimensional family, this time with parameters $b_1'$, $b_2'$, and $h(c_1) - h(c_2)$. So if we can control curve multiplicities, then we can apply Theorem \ref{thm:sharir--zahl} to bound $I(\Pi, \Gamma)$. As before, the equation of a curve in $\Gamma$ uniquely determines $b_1'$, $b_2'$, and $h(c_1) - h(c_2)$ (this is where we use the assumption that $\deg h \geq 2$); this means all curve multiplicities are $O(\abs{B})$.

We again dyadically partition $\Gamma$ by curve multiplicity --- for each $m \in \{2^0, 2^1, \ldots\}$, we let $\Gamma_m$ be the portion of $\Gamma$ consisting of curves with multiplicities in $[m, 2m)$; then $I(\Pi, \Gamma) = \sum_m I(\Pi, \Gamma_m)$. For each $m$, we let $\Gamma_m'$ be the variant of $\Gamma_m$ where we include each curve only once. Then $I(\Pi, \Gamma_m) = \Theta(m \cdot I(\Pi, \Gamma_m'))$, and we can use Theorem \ref{thm:sharir--zahl} to bound $I(\Pi, \Gamma_m')$; this gives
\begin{equation}
    I(\Pi, \Gamma_m) = O(m\abs{\Pi}^{6/11}\sabs{\Gamma_m'}^{9/11 + \eta} + m\abs{\Pi}^{2/3}\sabs{\Gamma_m'}^{2/3} + m\abs{\Pi} + m\sabs{\Gamma_m'}).\label{eqn:incidence-bound-rtr}
\end{equation}
We will deal with the sum of each term over $m$ in the same way as in Section \ref{sec:rt-upper}. For the fourth term of \eqref{eqn:incidence-bound-rtr}, we have \[\sum_m m\sabs{\Gamma_m'} \leq \sum_m \sabs{\Gamma_m} = \abs{\Gamma} = O\left(\frac{st^2}{r}\abs{B}^4\right),\] giving the term $\term{IV}$ of Lemma \ref{lem:rtr-upper}. When summing the third term of \eqref{eqn:incidence-bound-rtr} over $m$, we obtain a geometric sequence with largest element $O(\abs{B}\abs{\Pi})$. So in total, this contributes \[O(\abs{B}\abs{\Pi}) = O\left(\frac{t}{r}\abs{B}\abs{B + B}^2\right),\] giving the term $\term{III}$ of Lemma \ref{lem:rtr-upper}. For the first and second terms of \eqref{eqn:incidence-bound-rtr}, we note that $m\sabs{\Gamma_m'} \leq \abs{\Gamma} = O(\frac{st^2}{r}\abs{B}^4)$ and \[m^2\sabs{\Gamma_m'} = O\left(\frac{t}{r}\abs{B}^2\sabs{\cS_t}\right).\] The latter bound is obtained by repeating the proof of Claim \ref{claim:mult-to-qth}, considering pairs of $4$-tuples $(b_1', b_2', c_1, c_2)$ and $(b_3', b_4', c_3, c_4)$ corresponding to the same curve. Then by repeating the computations in the respective part of Section \ref{sec:rt-upper}, the sums of the first and second terms of \eqref{eqn:incidence-bound-rtr} over $m$ give the terms $\term{I}$ and $\term{II}$ of Lemma \ref{lem:rtr-upper}, respectively. 

\subsection{Finishing the proof}\label{subsec:qth-upper-conclusion}

Finally, in this subsection we put together Lemmas \ref{lem:r-and-s}, \ref{lem:rtr-lower}, and \ref{lem:rtr-upper} to prove Lemma \ref{lem:qth-upper}, our upper bound on $\sabs{\cS_t}$. First, combining Lemmas \ref{lem:r-and-s} and \ref{lem:rtr-lower} gives \[\sabs{\cT_{t, r}} = \Omega\left(\frac{t}{r}\abs{B}^{2 - \eta}\sabs{\cS_t}\right).\] We now combine this with the upper bound on $\sabs{\cT_{t, r}}$ from Lemma \ref{lem:rtr-upper}, splitting into cases according to the largest term of this upper bound. 

\emph{Case 1 (term $\term{I}$ is largest).} In this case, we get \[\frac{t}{r}\abs{B}^{2 - \eta}\sabs{\cS_t} = O\left(\frac{s^{7/11}t^2}{r^{15/11}}\abs{B}^{32/11 + 5\eta}\abs{B + B}^{12/11}\sabs{\cS_t}^{2/11}\right),\] which rearranges to \[\sabs{\cS_t} = O\left(\frac{s^{7/9}t^{11/9}}{r^{4/9}}\abs{B}^{10/9 + 8\eta}\abs{B + B}^{4/3}\right).\] And since $r \geq s$ and $r, s \leq 1$, we have $s^{7/9} \leq r^{4/9}$; so this means \[\sabs{\cS_t} = O(t^{11/9}\abs{B}^{10/9 + 8\eta}\abs{B + B}^{4/3}),\] giving the bound in Lemma \ref{lem:qth-upper}. 

\emph{Case 2 (term $\term{II}$ is largest).} In this case, a similar computation gives \[\sabs{\cS_t} = O\left(\frac{s^{1/2}t}{r^{1/2}}\abs{B}^{3\eta}\abs{B + B}^2\right) = O(t\abs{B}^{3\eta}\abs{B + B}^2).\] Using the facts that $t\abs{B} \geq 1$ and $\abs{B + B} \leq \abs{B}^{4/3}$, we can check that this is smaller than the bound from Case 1, as \[t\abs{B}^{3\eta}\abs{B + B}^2 \leq t^{11/9}\abs{B}^{2/9 + 8\eta}\abs{B + B}^2 \leq t^{11/9}\abs{B}^{10/9 + 8\eta}\abs{B + B}^{4/3}.\] So Lemma \ref{lem:qth-upper} is true in this case as well. 

\emph{Case 3 (term $\term{III}$ is largest).} In this case, a similar computation gives \[\sabs{\cS_t} = O(\abs{B}^{-1 + \eta}\abs{B + B}^2),\] and since $t\abs{B} \geq 1$, this is smaller than the bound from Case 2. 

\emph{Case 4 (term $\term{IV}$ is largest).} In this case, a similar computation gives \[\sabs{\cS_t} = O(st\abs{B}^{2 + \eta}).\] Since $s \leq 1$ and $\abs{B} \leq \abs{B + B}$, the above bound is also smaller than the bound from Case 2.

We conclude that Lemma \ref{lem:qth-upper} is true in all four cases. 

\section{Completing the proofs}\label{sec:conclusion}

We now perform the final computations needed to complete the proofs of Theorems \ref{thm:expand-one-set} and \ref{thm:expand-two-sets}. 

\subsection{Setting the value of \texorpdfstring{$t$}{t}}\label{subsec:set-t}

First, we need to set the value of $t$ that we will plug into our lemmas. Before doing so, we note that there is at most one $a \in \RR$ for which $g(a + p(y)) + h(y)$ is constant (as a polynomial in $y$). By removing this value from $A$ if it is originally present (i.e., applying Theorems \ref{thm:expand-one-set} and \ref{thm:expand-two-sets} to $A \setminus \{a\}$), we can assume $g(a + p(y)) + h(y)$ is nonconstant for all $a \in A$. 

Our goal is to get an upper bound on $\abs{\cQ}$ by combining the lower bound on $\sabs{\cR_t}$ from Lemma \ref{lem:rt-lower} with the upper bound on $\sabs{\cR_t}$ from Lemma \ref{lem:rt-upper}. For this, we need the first term on the right-hand side of Lemma \ref{lem:rt-lower} to dominate the second. Other than this, we would like $t$ to be as small as possible. So we set 
\begin{equation}
    t = C \cdot \frac{\abs{A}\abs{B}}{\abs{\cQ}}\label{eqn:set-t}
\end{equation} 
for a sufficiently large constant $C$ (depending on the implicit constants in Lemma \ref{lem:rt-lower}). 

As mentioned in Section \ref{sec:outline}, we need to check that $t \leq 1$ and $t\abs{A}, t\abs{B} \geq 1$. 
\begin{itemize}
    \item To check that $t\abs{A} \geq 1$, we claim that $\abs{\cQ} = O(\abs{A}^2\abs{B})$. To see this, imagine constructing a $4$-tuple $(a_1, a_2, b_1, b_2) \in \cQ$ by first choosing $a_1$, $a_2$, and $b_1$. Then we know the value of $f(a_2, b_2) = g(a_2 + p(b_2)) + h(b_2)$; and $g(a_2 + p(y)) + h(y)$ is a nonconstant polynomial in $y$ of degree at most $d$, so there are at most $d = O(1)$ choices for $b_2$. 
    
    Then as long as the constant $C$ in \eqref{eqn:set-t} is sufficiently large, we indeed have $t\abs{A} \geq 1$. 
    \item Similarly, to check that $t\abs{B} \geq 1$, we claim that $\abs{\cQ} = O(\abs{A}\abs{B}^2)$. To see this, if we first choose $a_1$, $b_1$, and $b_2$, then we know the value of $f(a_2, b_2) = g(a_2 + p(b_2)) + h(b_2)$, which means there are at most $d = O(1)$ choices for $a_2$. 
    \item Finally, it is not necessarily true that this choice of $t$ satifies $t \leq 1$. But if $t > 1$, then we have $\abs{\cQ} = O(\abs{A}\abs{B})$, and using Cauchy--Schwarz as in \eqref{eqn:cauchy--schwarz} gives \[\abs{f(A, B)} \geq \frac{\abs{A}^2\abs{B}^2}{\abs{\cQ}} = \Omega(\abs{A}\abs{B}).\] This is where the final terms of Theorems \ref{thm:expand-one-set} and \ref{thm:expand-two-sets} come from.
\end{itemize}

\subsection{Proof of Theorem \ref{thm:expand-one-set}}\label{subsec:one-set-concl}

We now perform the final computations for Theorem \ref{thm:expand-one-set}. For our value of $t$, Lemma \ref{lem:rt-lower} gives that $\sabs{\cR_t} = \Omega(t^2\abs{A}^2\sabs{\cQ})$, while Lemma \ref{lem:rt-upper} gives an \emph{upper} bound on $\sabs{\cR_t}$ in terms of $\sabs{\cS_t}$. Since in the setting of Theorem \ref{thm:expand-one-set} we do not know anything about the additive structure of $B$, we simply use the crude bound \[\sabs{\cS_t} = O(t\abs{B}^3).\] We then plug this into Lemma \ref{lem:rt-upper} and combine the result with the lower bound from Lemma \ref{lem:rt-lower}, performing casework on which term from Lemma \ref{lem:rt-upper} is largest. 

\emph{Case 1 (term $\term{I}$ is largest).} In this case, we get that \[t^2\abs{A}^2\abs{\cQ} = O(t^{24/11}\abs{A}^{18/11 + 2\eta}\abs{B}^{20/11 + 3\eta}\abs{A + A}^{12/11}).\] Rearranging and plugging in our value of $t$ gives \[\abs{\cQ} = O(t^{2/11}\abs{A}^{-4/11 + 2\eta}\abs{B}^{20/11 + 3\eta}\abs{A + A}^{12/11}) = O\left(\frac{\abs{A}^{-2/11 + 2\eta}\abs{B}^{2 + 3\eta}\abs{A + A}^{12/11}}{\abs{\cQ}^{2/11}}\right),\] and this rearranges to \[\abs{\cQ} = O(\abs{A}^{-2/13 + 2\eta}\abs{B}^{22/13 + 3\eta}\abs{A + A}^{12/13}).\] Finally, applying Cauchy--Schwarz as in \eqref{eqn:cauchy--schwarz} gives that \[\abs{f(A, B)} \geq \frac{\abs{A}^2\abs{B}^2}{\abs{\cQ}} = \Omega\left(\frac{\abs{A}^{28/13 - 2\eta}\abs{B}^{4/13 - 3\eta}}{\abs{A + A}^{12/13}}\right),\] giving the first term of Theorem \ref{thm:expand-one-set}. 

\emph{Case 2 (term $\term{II}$ is largest).} In this case, we get \[t^2\abs{A}^2\abs{\cQ} = O(t^2\abs{A}^{4/3}\abs{B}^{5/3 + \eta}\abs{A + A}^{4/3}),\] which rearranges to \[\abs{\cQ} = O(\abs{A}^{-2/3}\abs{B}^{5/3 + \eta}\abs{A + A}^{4/3}).\] After applying Cauchy--Schwarz, this gives the second term in Theorem \ref{thm:expand-one-set}. 

\emph{Case 3 (term $\term{III}$ is largest).} In this case, we get \[t^2\abs{A}^2\sabs{\cQ} = O(t\abs{A + A}^2\abs{B}).\] By \eqref{eqn:set-t}, we have $t\abs{\cQ} = \Theta(\abs{A}\abs{B})$, so $\abs{A}^3 = O(\abs{A + A}^2)$. This contradicts the assumption of Theorem \ref{thm:expand-one-set} that $\abs{A + A} \leq c\abs{A}^{3/2}$ (if $c$ is sufficiently small), so this case cannot occur. 

\emph{Case 4 (term $\term{IV}$ is largest).} In this case, we get $\abs{\cQ} = O(\abs{B}^2)$, which by Cauchy--Schwarz gives $\abs{f(A, B)} = \Omega(\abs{A}^2)$, corresponding to the third term in Theorem \ref{thm:expand-one-set}. 

We conclude that Theorem \ref{thm:expand-one-set} holds in all cases. 

\subsection{Proof of Theorem \ref{thm:expand-two-sets}}\label{subsec:two-sets-concl}

Finally, we perform the computations for Theorem \ref{thm:expand-two-sets}. We again have $\sabs{\cR_t} = \Omega(t^2\abs{A}^2\abs{\cQ})$, while this time when using the upper bound on $\sabs{\cR_t}$ from Lemma \ref{lem:rt-upper}, we plug in the upper bound on $\sabs{\cS_t}$ from Lemma \ref{lem:qth-upper}. We can again perform casework on which term in the resulting bound is largest. The cases where term $\term{III}$ or $\term{IV}$ is the largest is the same as in Subsection \ref{subsec:one-set-concl} --- the case where $\term{III}$ is largest gives a contradiction, and the case where $\term{IV}$ is largest gives $\abs{f(A, B)} = \Omega(\abs{A}^2)$. 

In the case where term $\term{I}$ is largest, we get \[t^2\abs{A}^2\abs{\cQ} = O(t^{20/9}\abs{A}^{18/11 + 2\eta}\abs{B}^{146/99 + 5\eta}\abs{A + A}^{12/11}\abs{B + B}^{8/33}).\] After plugging in the value of $t$, rearranging to isolate $\abs{\cQ}$, and repeating the above Cauchy--Schwarz argument, we obtain the first term of the bound of Theorem \ref{thm:expand-two-sets}. 

Similarly, when term $\term{II}$ is largest, we get \[t^2\abs{A}^2\abs{\cQ} = O(t^{56/27}\abs{A}^{4/3}\abs{B}^{28/27 + 4\eta}\abs{A + A}^{4/3}\abs{B + B}^{4/9}),\] which ends up giving the second term of Theorem \ref{thm:expand-two-sets}. So Theorem \ref{thm:expand-two-sets} is true in all cases. 

\section{Additional proofs}

In this section, we present our last two remaining proofs: the proofs of Proposition \ref{prop:asquared} and Corollary \ref{cor:heavy-lines}.

\subsection{Proof of Proposition \ref{prop:asquared}}\label{subsec:asquared}

In this subsection, we prove our almost tight lower bound for $A^2 + A^2$ when $A$ has small sumset. Our proof is based on ideas of Elekes and Ruzsa \cite{ER03}. The idea is to show that 
\begin{equation}
    \sabs{A^2 + A^2} \cdot \max\{\abs{A - A + A}^4, \abs{A + A + A}^4\} \cdot \log \abs{A} = \Omega(\abs{A}^6).\label{eqn:asquared-desired}
\end{equation} 
When $\abs{A + A} \leq K\abs{A}$, Pl\"unnecke's inequality (see \cite[Theorem 7.3.3]{Zha23}, for example) gives that $\abs{A - A + A} \leq K^3\abs{A}$ and $\abs{A + A + A} \leq K^3\abs{A}$. Plugging these bounds into \eqref{eqn:asquared-desired} gives \[\sabs{A^2 + A^2} = \Omega\left(\frac{\abs{A}^2}{K^{12}\log \abs{A}}\right).\] So in order to prove Proposition \ref{prop:asquared}, it suffices to prove \eqref{eqn:asquared-desired}. 

First, we may assume that $0 \in A$, by replacing $A$ with $A \cup \{0\}$ otherwise (this increases the left-hand side of \eqref{eqn:asquared-desired} by at most a constant factor). Then $A \subseteq A - A + A$ and $A \subseteq A + A + A$. 

Let $\cQ = \{(a, b, c, d) \in A^4 \mid a^2 + b^2 = c^2 + d^2\}$. By Cauchy--Schwarz, we have 
\begin{equation}
    \abs{\cQ} \geq \frac{\abs{A}^4}{\sabs{A^2 + A^2}}.\label{eqn:cauchy-q}\end{equation} 
To deduce \eqref{eqn:asquared-desired}, we will combine this with an upper bound on $\abs{\cQ}$. The idea is that for every $(a, b, c, d) \in \cQ$, we have $(a - c)(a + c) = (d - b)(d + b)$. This implies that for all $x, y \in A$, the three points \[(x, y), \, (x + a - c, y + d + b), \, (x + d - b, y + a + c)\] form a collinear triple in $(A + A - A) \times (A + A + A)$, so the number of collinear triples in this set is at least $\abs{A}^2\abs{\cQ}$. 

On the other hand, by Szemer\'edi--Trotter, the number of collinear triples in a set of $n$ points with at most $\ell$ collinear is $O(n^2\log \ell + \ell^2 n)$ (see for example \cite[Lemma 2.2]{ER03}). 

In our case, we have $n = \abs{A + A - A}\abs{A + A + A} \leq \max\{\abs{A + A - A}^2, \abs{A + A + A}^2\}$ and $\ell = \max\{\abs{A + A - A}, \abs{A + A + A}\}$, so this bound gives \[\abs{A}^2\abs{\cQ} \leq \max\{\abs{A - A + A}^4, \abs{A + A + A}^4\}\log \abs{A}.\] Combining this upper bound on $\abs{\cQ}$ with the lower bound from \eqref{eqn:cauchy-q} implies \eqref{eqn:asquared-desired}.

\subsection{Proof of Corollary \ref{cor:heavy-lines}}\label{subsec:heavy}

In this subsection, we prove our slight improvement to the results of \cite{SZdZ16} and \cite{RRNS15} that a set of points with few distinct distances cannot have too many points on one line. Our proof combines the ideas from \cite{SZdZ16} and \cite{RRNS15} with Corollary \ref{cor:distances-one-set}.

Let $\cP$ be a set of points as in the statement of Corollary \ref{cor:heavy-lines}, and let $\ell$ be a line that contains $m$ points of $\cP$. We may assume that $m \geq n^{4/5}$, since otherwise there is nothing to prove (as $\frac{13}{16} > \frac{4}{5}$). Let $\cP_1 = \ell \cap \cP$. By translating and rotating $\RR^2$, we may assume that $\ell$ is the $x$-axis. We denote the coordinates of a point $p \in \RR^2$ by $(p_x, p_y)$. Finally, as in the previous proofs, we let $\eta$ denote an error parameter which we will eventually take to be small with respect to $\eps$.

The proof works by considering a certain energy and upper-bounding it using incidence bounds. Specifically, we define the sets
\begin{align*}
    \cQ &= \{(a, b, p, q) \in \cP_1^2 \times \cP^2 \mid d(a, p) = d(b, q)\}, \\
    \cQ' &= \{(a, b, p, q) \in \cQ \mid p_y \neq \pm q_y\}.
\end{align*}
By Cauchy--Schwarz, we have \[\abs{\cQ} \geq \frac{\abs{\cP_1}^2\abs{\cP}^2}{\abs{\Delta(\cP_1, \cP)}} \geq \frac{m^2n^2}{\frac{1}{5}n} = 5m^2n.\] To convert this into a lower bound on $\sabs{\cQ'}$, note that when constructing a $4$-tuple $(a, b, p, q) \in \cQ \setminus \cQ'$, there are $m^2n$ choices for $a$, $b$, and $p$. Then the condition $p_y = \pm q_y$ means there are at most $2$ choices for $q_y$, and the condition $d(a, p) = d(b, q)$ means that each leads to at most two choices for $q_x$. This implies that $\sabs{\cQ \setminus \cQ'} \leq 4m^2n$, so \[\sabs{\cQ'} = \abs{\cQ} - \sabs{\cQ \setminus \cQ'} \geq 5m^2n - 4m^2n = m^2n.\] 

We will now get an \emph{upper} bound for $\sabs{\cQ'}$ by setting up a collection of points $\Pi$ and curves $\Gamma$ whose incidences correspond to elements of $\cQ'$ and applying the incidence bound of Theorem \ref{thm:sharir--zahl}. We define \[\Pi = \{(a_x, b_x) \mid (a, b) \in \cP_1^2\}.\] We construct a multiset of curves $\Gamma$ as follows. For each $(p, q) \in \cP^2$ with $p_y \neq \pm q_y$, we add to $\Gamma$ the curve defined by \[(x - p_x)^2 + p_y^2 = (y - q_x)^2 + q_y^2.\] Then we have $\abs{\Pi} = m^2$ and $\abs{\Gamma} \leq n^2$, and $\sabs{\cQ'} = I(\Pi, \Gamma)$ by definition. 

The curves in $\Gamma$ belong to a $3$-dimensional family defined by parameters $p_x$, $q_x$, and $p_y^2 - q_y^2$. The condition $p_y \neq \pm q_y$ implies that these curves are irreducible. So to apply Theorem \ref{thm:sharir--zahl}, it remains to obtain an upper bound on the curve multiplicities. 

If two pairs $(p, q)$ and $(p', q')$ correspond to the same curve, then $p_x = p_x'$ and $q_x = q_x'$. So to control curve multiplicities, we let $k$ be the maximum number of points of $\cP$ on a line $\ell'$ orthogonal to $\ell$; then all curves in $\Gamma$ have multiplicity at most $k$. 

We first need to show that $k$ is not too large. To do so, let $\ell'$ be a line orthogonal to $\ell$ containing $k$ points, and let $\cP_2 = \cP \cap \ell'$. Both \cite{SZdZ16} and \cite{RRNS15} use additive combinatorics estimates to show that if $k$ is large, then either there are many distances \emph{within} $\cP_1$, or there are many distances \emph{between} $\cP_1$ and $\cP_2$ (contradicting the fact that there are not too many distances in $\cP$). Corollary \ref{cor:distances-one-set} provides a quantitative improvement to these additive combinatorics estimates in the relevant regime. By using it, we obtain the following bound. 

\begin{claim}\label{claim:k-bound}
    We have $k = O(\frac{n^{25/4 + 8\eta}}{m^7} + \frac{n}{m})$. 
\end{claim}

\begin{proof}
    Since $\abs{\cP_1} = m \geq n^{4/5}$ and $\abs{\Delta(\cP_1)} \leq n$, we have $\abs{\Delta(\cP_1)} \leq \abs{\cP_1}^{5/4}$. This means we can apply Corollary \ref{cor:distances-one-set} to the sets $\cP_1 \subseteq \ell$ and $\cP_2 \subseteq \ell'$, which gives \[\abs{\Delta(\cP_1, \cP_2)} = \Omega\left(\min\left\{\frac{m^{28/13 - \eta}k^{4/13 - \eta}}{\abs{\Delta(\cP_1)}^{12/13}}, m^2, km\right\}\right).\] Both $\abs{\Delta(\cP_1, \cP_2)}$ and $\abs{\Delta(\cP_1)}$ are at most $n$, so \[n = \Omega\left(\min\left\{\frac{m^{28/13 - \eta}k^{4/13 - \eta}}{n^{12/13}}, m^2, km\right\}\right) = \Omega\left(\min\left\{\frac{m^{28/13}k^{4/13}}{n^{12/13 + 2\eta}}, m^2, km\right\}\right).\] We partition the analysis into cases according to the smallest term inside the minimum. If the first term is smallest, then we get \[k = O\left(\frac{n^{25/4 + 8\eta}}{m^7}\right).\] If the second term is smallest, then we get $m = O(n^{1/2})$, which contradicts the assumption that $m \geq n^{4/5}$. Finally, if the third term is smallest, then we get $k = O(\frac{n}{m})$. 
\end{proof}

Now we bound $\sabs{\cQ'} = I(\Pi, \Gamma)$ by dyadically partitioning $\Gamma$ based on curve multiplicity and applying Theorem \ref{thm:sharir--zahl} to each chunk of this partition. For each $j \in \{2^0, 2^1, \ldots\}$, let $\Gamma_j$ be the multiset consisting of curves whose multiplicities in $\Gamma$ are in $[j, 2j)$, each with the same multiplicity as in $\Gamma$, so that $I(\Pi, \Gamma) = \sum_j I(\Pi, \Gamma_j)$. Let $\Gamma_j'$ be the set of curves of $\Gamma_j$, each appearing just once. Then using Theorem \ref{thm:sharir--zahl} to bound $I(\Pi, \Gamma_j')$ gives 
\begin{align*}
    I(\Pi, \Gamma_j) &= \Theta(j \cdot I(\Pi, \Gamma_j')) \\
    &= O(j\abs{\Pi}^{6/11}\sabs{\Gamma_j'}^{9/11 + \eta} + j\abs{\Pi}^{2/3}\sabs{\Gamma_j'}^{2/3} + j\abs{\Pi} + j\sabs{\Gamma_j'}) \\ 
    &= O(j^{2/11}\abs{\Pi}^{6/11}\abs{\Gamma}^{9/11 + \eta} + j^{1/3}\abs{\Pi}^{2/3}\abs{\Gamma}^{2/3} + j\abs{\Pi} + \sabs{\Gamma_j}),
\end{align*} 
where in the last transition we used the fact that $j\sabs{\Gamma_j'} \leq \sabs{\Gamma_j} \leq \abs{\Gamma}$. When we sum over $j$, the total contribution of the fourth term is at most $\sum_j \abs{\Gamma_j} = \abs{\Gamma}$. Each of the first three terms becomes a geometric series whose sum is dominated by its last term, where $j = O(k)$. So we get \[\sabs{\cQ'} \leq I(\Pi, \Gamma) = \sum_j I(\Pi, \Gamma_j) = O(k^{2/11}\abs{\Pi}^{6/11}\abs{\Gamma}^{9/11 + \eta} + k^{1/3}\abs{\Pi}^{2/3}\abs{\Gamma}^{2/3} + k\abs{\Pi} + \abs{\Gamma}).\] 

Now plugging in $\abs{\Pi} = m^2$ and $\abs{\Gamma} \leq n^2$ and combining this with our lower bound $\sabs{\cQ'} \geq m^2n$, we get that 
\begin{equation*}
    m^2n = O(k^{2/11}m^{12/11}n^{18/11 + \eta} + k^{1/3}m^{4/3}n^{4/3} + km^2 + n^2).
\end{equation*}
By Claim \ref{claim:k-bound} and the assumption $m \geq n^{4/5}$, the third and fourth terms on the right-hand side are $o(m^2n)$, so we may ignore them. Then plugging our bound on $k$ from Claim \ref{claim:k-bound} into the first two terms gives \[m^2n = O\left(\frac{n^{61/22 + 3\eta}}{m^{2/11}} + m^{10/11}n^{20/11 + \eta} + \frac{n^{41/12 + 3\eta}}{m} + mn^{5/3}\right).\] Finally, performing casework on which term on the right-hand side is largest and rearranging gives \[m = O(n^{13/16 + 2\eta} + n^{3/4 + \eta} + n^{29/36 + \eta} + n^{2/3}) = O(n^{13/16 + 2\eta}),\] completing the proof of Corollary \ref{cor:heavy-lines}. 

\section*{Acknowledgements}

The second and third authors thank Caleb Ji for conversations about this problem. 

\bibliographystyle{alpha}
\bibliography{refs}
\end{document}